\documentclass[11pt,a4paper]{amsart}

\usepackage[utf8]{inputenc}

\title{Norm estimates for the $\dbar$-equation on a non-reduced space}

\author{Mats Andersson \& Richard L\"ark\"ang}
\thanks{The authors were partially supported by grants from the Swedish Research Council.}
\subjclass[2000]{32A26, 32A27, 32B15, 32W05}
\address{Department of Mathematical Sciences, Division of Algebra and Geometry,
Chalmers University of Technology and the  University of Gothenburg,
SE-412 96 G\"{o}teborg, Sweden}
\email{matsa@chalmers.se, larkang@chalmers.se}

\usepackage{amsmath}
\usepackage{amsthm,amssymb,latexsym}
\usepackage{url}
\usepackage{mathrsfs}
\usepackage{graphicx}
\usepackage{geometry}
\newtheorem{thm}{Theorem}[section]
\newtheorem{lma}[thm]{Lemma}
\newtheorem{cor}[thm]{Corollary}
\newtheorem{prop}[thm]{Proposition}

\theoremstyle{definition}

\newtheorem{df}[thm]{Definition}

\theoremstyle{remark}

\newtheorem{preremark}[thm]{Remark}
\newtheorem{preex}[thm]{Example}

\newenvironment{remark}{\begin{preremark}}{\qed\end{preremark}}
\newenvironment{ex}{\begin{preex}}{\qed\end{preex}}

\newcommand{\C}{\mathbb{C}}
\newcommand{\debar}{\bar{\partial}}
\newcommand{\dbar}{\bar{\partial}}
\newcommand{\A}{\mathscr{A}}

\newcommand{\J}{\mathcal{J}}
\newcommand{\I}{\mathcal{I}}
\newcommand{\E}{\mathscr{E}}

\newcommand{\W}{\mathcal{W}}

\newcommand{\Ok}{\mathscr{O}}

\newcommand{\K}{\mathcal{K}}
\newcommand{\F}{\mathcal{F}}

\newcommand{\om}{{\scalebox{1.3}{$\omega$}}}
\newcommand{\V}{{\mathcal V}}

\newcommand{\CH}{\mathcal{{ CH}}}

\newcommand{\pmm}{pseudomeromorphic }
\newcommand{\nbh}{neighborhood }
\newcommand{\1}{{\bf 1}}
\newcommand{\w}{\wedge}
\newcommand{\codim}{{\text{codim}\,}}

\newcommand{\Homs}{{\mathcal Hom}}
\newcommand{\Hom}{{\text{Hom}\,}}

\newcommand{\Kers}{{\mathcal Ker\,}}
\newcommand{\N}{{\mathcal N}}

\newcommand{\U}{{\mathcal U}}
\newcommand{\Cu}{{\mathcal C}}
\newcommand{\La}{{\mathcal L}}

\DeclareMathOperator{\supp}{supp}
\DeclareMathOperator{\rank}{rank}
\DeclareMathOperator{\End}{End}

\numberwithin{equation}{section}

\usepackage{xcolor}

\date{\today}

\begin{document}

\nocite{*}
\bibliographystyle{plain}

\begin{abstract}
    We study norm-estimates for the $\dbar$-equation on non-reduced analytic spaces.
    Our main result is that on a non-reduced analytic space, which is Cohen-Macaulay
    and whose underlying reduced space is smooth, the $\dbar$-equation for $(0,1)$-forms
    can be solved with $L^p$-estimates.
\end{abstract}

\maketitle
\thispagestyle{empty}

\section{Introduction}

Various  estimates for solutions of the $\dbar$-equation on a smooth complex manifold are known since long ago.
The paramount methods are the $L^2$-methods, going back to H\"ormander, Kohn, and others, and
integral representation formulas, first used by Henkin and by Skoda.   Starting with \cite{PaSt1,PaSt2} there has been an
increasing interest for $L^2$- and  $L^p$-estimates for $\dbar$ on non-smooth reduced analytic spaces in later years, see,
e.g., \cite{BS,FOV,LR2,OvVass2,Rupp2}.
In \cite{AA, CDM,De} are results about $L^2$-estimates of extensions from non-reduced subvarieties.
In this paper we try to initiate the study of $L^p$-estimates for the $\dbar$-equation on
a non-reduced analytic space.

\smallskip
Let $X$ be an analytic space of pure dimension $n$ with structure sheaf $\Ok_X$.
Locally then we have an embedding $i\colon X\to \U \subset \C^N$ and a coherent
ideal sheaf $\J\subset \Ok_\U$ of pure dimension $n$ such that
$\Ok_X=\Ok_\Omega/\J$ in $X\cap \U$.  In \cite{AL} we introduced a notion of smooth $(0,*)$-forms on $X$ and
proved that if the underlying reduced space $X_{red}$ is smooth and in addition $\Ok_X$
is Cohen-Macaulay,  then there is a smooth solution to $\dbar u=\phi$ if $\phi$ is smooth and
$\dbar\phi=0$.  More generally we defined sheaves $\A_X^q$ of $(0,q)$-forms on $X$ that
are closed under multiplication by smooth $(0,*)$-forms and coincides with $\E_X^{0,q}$
where $X_{red}$ is smooth and $\Ok_X$ is Cohen-Macaulay, such that
$$
0\to \Ok_X\to \A_X^0\to \cdots \to \A_X^n\to 0
$$
is a fine resolution of $\Ok_X$.   The solutions to the $\dbar$-equation are obtained by intrinsic
integral formulas on $X$.
Variants of the $\dbar$-equation on non-reduced spaces have also been studied by Henkin-Polyakov,
\cite{HePo2,HePo3}.

In \cite{Anorm} was introduced a pointwise norm $|\cdot|$ on forms $\phi\in \E_X^{0,*}$. That is,
$|\phi(x)|_X$ is non-negative function on $X_{red}$ which vanishes in a \nbh of a point $x_0$ if and only if
$\phi$ vanishes there.
It was proved that $\Ok_X$ is complete with respect to the topology of uniform convergence on compacts induced
by this norm.
In this paper we will only discuss spaces where $X_{red}$ is smooth.
In \cite{Anorm} is defined an intrinsic  coherent left $\Ok_{X_{red}}$-module  $\N_X$ of differential operators
$\Ok_X\to\Ok_{X_{red}}$, and the pointwise norm is defined as
$$
|\phi(x)|^2_X=\sum_k|(\La_k\phi)(x)|^2,
$$
where $\La_j$ is a finite set of local generators for $\N_X$. Clearly another set of generators will give
rise to an equivalent norm. In particular, it becomes meaningful to say that $\phi$ vanishes at a point
$x\in X$.  The norm  extends to smooth $(0,*)$-forms.
By a partition of unity we  patch together and define a fixed global $|\cdot|_X$.
Let us also choose a volume element $dV$ on $X_{red}$.
We define $L^p_{0,*;X}$ as the (local) completion of $\E^{0,*}_X$ with respect to
the  $L^p$-norm. In the same way we define $C_{0,*; X}$ as the completion with respect to
the uniform norm.
Our main result is

\begin{thm}\label{main}
Let $X$ be an analytic space such that  $\Ok_X$ is Cohen-Macaulay and $X_{red}$ is smooth.
Assume that $1\le p<\infty$.
Given a point $x$ there are neighborhoods
$\V'\subset\subset\V\subset X$ and a constant $C_p$ such that if $\phi\in L^p_{0,1}(\V)$ and $\dbar\phi=0$,  then there is
$\psi \in L^p_{0,0}(\V')$ such that $\dbar \psi=\phi$ and
$$
\int_{\V'_{red}}|\psi|_X^p dV\le C_p^p \int_{\V_{red}}|\phi|_X^p dV.
$$
%
Moreover, there is a constant $C_\infty$ such that
if $\phi\in C_{0,1}(\V)$ and $\dbar\phi=0$,  then there is a solution $\psi\in C_{0,0}(\V')$ such that
$$
\sup_{\V'_{red}}|\psi|_X\le C_\infty \sup_{\V_{red}}|\phi|_X.
$$
\end{thm}

By standard sheaf theory, and the fact that
$\psi \in L^p_{0,0;X}$ and $\dbar\psi=0$ implies that  $\psi \in \Ok_X$, see Lemma~\ref{lma:dbarClosed} and \eqref{inka}, we get the following corollaries.

\begin{cor}
Assume that $X$ is a compact analytic space such that $\Ok_X$ is Cohen-Macaulay, $X_{red}$ is smooth.
If  $\phi\in L^p_{0,1}(X)$, $\dbar\phi=0$  and
the cohomology class of $\phi$ in $H^1(X,\Ok_X)$ vanishes, then there is
$\psi \in L^p_{0,0}(X)$ such that $\dbar \psi=\phi$.
If $\phi \in C_{0,1}(X)$, then there is a solution $\psi\in C_{0,0}(X)$.
\end{cor}
Notice that $\phi$ defines a \v{C}ech cohomology class in $H^1(X,\Ok_X)$ through $(\psi_j-\psi_k)_{j,k}$,
where $(\psi_j)$ are local $\dbar$-solutions on a covering $\U_j$ of $X$.

\begin{cor}
Assume that $X$ is an Stein space such that $\Ok_X$ is Cohen-Macaulay and $X_{red}$ is smooth.
If $\phi\in L^p_{0,1,loc}(X)$ is $\dbar$-closed, then there is $\psi\in L^p_{0,0,loc}(X)$
such that $\dbar\psi=\phi$.
If $\phi\in C_{0,1}(X)$, then there is a solution $\psi\in C_{0,0}(X)$.
 \end{cor}

The proof of Theorem~\ref{main}  relies on the integral formulas in \cite{AL} in combination
with a new notion of sheaves of $\Cu_X^{0,*}$ of $(0,*)$-currents on $X$
which provide a fine resolution of $\Ok_X$ (see Section~\ref{intrinsic}). These sheaves should
have an independent interest.  In Remark~\ref{heuristic} we give a heuristic argument for
Theorem~\ref{main} which relies on these sheaves but with no reference to integral formulas.

We first consider a certain kind of
"simple" non-reduced space for which we prove these $L^p$-estimates for all $(0,q)$-forms, see Section~\ref{orm}.
We then prove the general case by means of a local embedding $X\to \hat X$, where $\hat X$ is simple. To
carry out the proof we need comparison results between the constituents in the integral formulas for
the two spaces. One of them is provided by \cite{LarComp}, whereas another one,
for the so-called Hefer mappings, is new,  see Section~\ref{comparison}.
The proof of Theorem~\ref{main} is in Section~\ref{sect:MainProof}.
Technical difficulties restrict us, for the moment,  to the case with $(0,1)$-forms.

We have no idea of whether one could prove Theorem~\ref{main},
e.g., in case $p=2$,  by $L^2$-methods.

The assumption that $X_{red}$ be smooth and $X$ be  Cohen-Macaulay  is crucial in this paper.
In the reduced case considerable difficulties appear already with
the presence of an isolated singularity;
besides the references already mentioned above,
see, e.g., \cite{DFV,FoGa, Nag, OvVass, RZ}.   In the non-reduced case even when $X_{red}$ is smooth, an isolated
non-Cohen-Macaulay point offers new difficulties. We discuss such an example in
Section~\ref{nonch}.

Throughout this paper $X$ is a non-reduced space of pure dimension $n$ and the underlying reduced
space $Z=X_{red}$ is smooth, if nothing else is explicitly stated.

\section{Some preliminaries}\label{prel}
Let $Y$ and $Y'$ be complex manifolds and
$f\colon Y'\to Y$ a proper mapping. If $\tau$ is a current on $Y$, then the push-forward, or direct image,
$f_*\tau$ is defined by the relation $f_*\tau.\xi=\tau.f^*\xi$ for test forms $\xi$.  If $\alpha$ is a smooth form on
$Y$, then we have the simple but useful relation
\begin{equation}\label{tomater}
\alpha\w f_*\tau=f_*( f^*\alpha\w \tau).
\end{equation}
In \cite{AW2,AS} was introduced the sheaf of
{\it pseudomeromorphic currents} on $Y$. 
Roughly speaking a \pmm current is the direct image under a holomorphic mapping of
a smooth form times a tensor product of one-variable principal value current
$1/z_j^m$ and $\dbar(1/z_k^{m'})$.
This sheaf is closed under $\dbar$ and under multiplication by smooth forms.
If a pseudomeromorphic current $\tau$ has support on a subvariety $V$ and the holomorphic function $h$ vanishes
on $V$, then $\bar h \tau =0$ and $d \bar  h\w \tau=0$. This leads to the crucial {\it dimension principle}

\begin{prop} \label{prop:dim}
    Let $\tau$ be a pseudomeromorphic current of bidegree $(*,q)$, and assume that
   the support of $\mu$ is contained in a subvariety of codimension $> q$. Then $\tau = 0$.
\end{prop}

We say that a current $a$ is {\it almost semi-meromorphic} in $Y$ if there is a modification $\pi\colon
Y'\to Y$, such that $a$ is the direct image of a form
$\alpha/f$, where $\alpha$ is smooth and $f$ is a holomorphic section of some line bundle on $Y'$.
Assume that $\mu$ is pseudomeromorphic, $a$ is an almost semi-meromorphic current,  $\chi_\epsilon = \chi(|F|^2/\epsilon)$,
where $\chi$ is a smooth cut-off function, and $F$ is a tuple of holomorphic functions such that $\{ F = 0 \}$ contains the set where $a$ is not smooth. Then the limit
\begin{equation} \label{eq:PMlimits}
    \lim_{\epsilon \to 0} \chi_\epsilon a \wedge \mu
  \end{equation}
 exists  and defines a  \pmm current $a\w\mu$ that is independent of the choice of $\chi$. We define
 $\dbar a\w\mu:=\dbar(a\w \mu)-(-1)^{\deg a} a\w \dbar \mu$. It is readily verified that
$  \dbar a\w \mu= \lim_{\epsilon \to 0} \dbar\chi_\epsilon \wedge a \wedge \mu$.
By Hironaka's theorem, any almost semi-meromorphic current is pseudomeromorphic.

\smallskip

Let $\U\subset\C^N$ be an open set, let $Z$ be a submanifold of dimension $n<N$, and
let $\kappa=N-n$.
The $\Ok_\U$-sheaf of Coleff-Herrera currents,
 $\CH_\U^Z$, see \cite{BjAbel}, consists of all $\dbar$-closed $(N,\kappa)$-currents in $\U$ with support on $Z$
that are annihilated by $\bar\J_Z$, i.e., by all $\bar h$ where $h$ is in $\J_Z$.  If $\J\subset\Ok_\U$ is
 an ideal sheaf with zero set $Z$, then  
$\Homs(\Ok_\U/\J,\CH_\U^Z)$ is the subsheaf of
$\mu$ in  $\CH_\U^Z$ that are annihilated by $\J$.  It is well-known that
$\Homs(\Ok_\U/\J,\CH_\U^Z)$ is coherent, cf.\ e.g., \cite[Theorem~1.5]{Aext}

\begin{remark} If $Z$ is not smooth, then $\CH_\U^Z$ is  defined in the same way, but one needs
an additional regularity condition, the so-called standard extension property, SEP, see, e.g., \cite[Section~2.1]{AL}.
When $Z$ is smooth, the currents in $\CH_\U^Z$ (with the definition given here) admit an expansion as in \cite[(3.4)]{ACH},
and so the SEP follows.
\end{remark}

Let us recall some properties of residue currents associated to a locally free resolution
\begin{equation}\label{krokus}
	0\to \Ok(E_{N_0})\stackrel{f_{N_0}}{\to} \Ok(E_{N_0-1})\cdots \stackrel{f_1}{\to} \Ok(E_0)\to 0
\end{equation}
of a coherent (ideal) sheaf $\Ok_{\U}/\J$.
The precise definitions and claimed results can all be found in \cite{AW1}.
Let us denote the complex \eqref{krokus} by $(E,f)$. Assume that the vector bundles $E_k$ are equipped with hermitian metrics.
The corresponding complex of vector bundles is pointwise exact on $\U\setminus Z$, where $Z=Z(\J)$.
There are associated currents $U$ and $R$. The current $U$ is almost semi-meromorphic on $\U$ and smooth on $\U\setminus Z$, and takes values in $\Hom(E,E)$.
The current $R$ is a pseudomeromorphic current on $\U$ that takes values in $\Hom(E_0,E)$ and has support on $Z$. One may write $R=\sum_k R_k$,
where $R_k$ is a $(0,k)$-current that takes values in $\Hom(E_0,E_k)$.
They satisfy the relation
\begin{equation}\label{plutt}
(f-\dbar)\circ U+U\circ (f-\dbar)= I_E- R.
\end{equation}
Here we use the compact notation $E=\oplus E_k$, $f=\sum f_k$.
By the dimension principle $R_k=0$ if $k < \kappa = \codim \J$.
In particular, since $(f-\dbar)^2=0$, it follows by \eqref{plutt} that $(f-\dbar)R=0$, so
\begin{equation} \label{eq:fR}
    f_{\kappa} R_\kappa = \dbar R_{\kappa-1} = 0.
\end{equation}
Moreover, $R$ is annihilated by $\bar\J_Z$, and it satisfies the duality principle
\begin{equation}\label{duality}
R \Phi=0  \ \   \text{if and only if} \ \  \Phi\in \J.
\end{equation}
We will typically assume that the resolution is chosen to be minimal at level $0$, i.e, such that $E_0 \cong \Ok$.
Thus, $\Hom(E_0,E_k) \cong E_k$, so we may consider $R_k$ as an $E_k$-valued current.
If $\Ok_\U/\J$ is Cohen-Macaulay, then we can choose the resolution so that $N_0=\kappa$.
Then it follows
that $R$ consists of the only term $R_\kappa$ that takes values in
$E_\kappa$, and from \eqref{plutt} that
$\dbar R_\kappa=0$.
We conclude that the components $\mu_1,\ldots,\mu_\rho$ of
$R_\kappa$, $\rho=\rank E_\kappa$,
are  in the sheaf $\Homs(\Ok_\U/\J,\CH_\U^Z)$.
It is proved in \cite[Example~1]{Aext} that these components $\mu_j$ actually generate
this sheaf.  It follows from \eqref{duality} that
\begin{equation}\label{duality1}
\mu\Phi=0, \ \mu\in  \Homs(\Ok_\U/\J,\CH_\U^Z), \ \   \text{if and only if} \ \  \Phi\in \J.
\end{equation}
By continuity \eqref{duality1} holds everywhere if $\J$ is has pure dimension.

\section{Pointwise norm on a non-reduced space $X$}\label{pnorm}
Recall that $X$ is  a non-reduced space of pure dimension $n$
with smooth underlying manifold $X_{red}=Z$.

Consider a local embedding $i\colon X\to \U\subset\C^N$ and assume that
$\pi \colon \U \to Z\cap\U$ is a submersion. Possibly after shrinking
$\U$ we can assume that we have coordinates
$(\zeta,\tau)=(\zeta_1,\ldots, \zeta_n, \tau_1,\ldots, \tau_\kappa)$ in $\U$
so that $Z\cap \U=\{ \tau=0\}$
and $\pi$ is the projection $(\zeta,\tau)\mapsto \zeta$.
Let $d\zeta=d\zeta_1\w\ldots, d\zeta_n$.

If $\mu$ is a section of $\Homs(\Ok_\U/\J,\CH_\U^Z)$ in $\U$, then
\begin{equation}\label{tomata}
\pi_*(\phi\mu)=: \La\phi \, d\zeta
\end{equation}
defines a holomorphic differential operator $\La\colon \Ok(X\cap\U)\to \Ok(Z\cap\U)$.
Following \cite[Section~1]{Anorm} we define $\N_X$ as the set of all such local operators $\La$ obtained from
some  $\mu$ in
$\Homs(\Ok_\U/\J,\CH_\U^Z)$ and a local submersion.
It follows from \eqref{tomater} and \eqref{tomata} that if $\xi$ is in $\Ok_Z$, then $\xi \La\phi=\La(\pi^*\xi \phi)$.
Thus $\N_X$ is a left $\Ok_Z$-module. It is coherent, in particular locally finitely generated,
and if $\La_j$ is a set of local generators, then $\phi=0$ if and only if  $\La_j\phi=0$ for all $j$, see
\cite[Theorem~1.3]{Anorm}.   If $\La_j$ is a finite set of local generators, therefore
\begin{equation}\label{atomat}
|\phi|^2_X=\sum_j |\La_j\phi|^2
\end{equation}
defines a local norm, and any other finite set of local generators gives rise to an equivalent local norm.

\begin{ex}\label{simple1}
Assume we have a local embedding and local coordinates $(\zeta,\tau)$ as above in $\U$.
Let  $M=(M_1,\ldots,M_\kappa)$ be a tuple of non-negative integers and consider the ideal
sheaf
$$
\I=\left< \tau_1^{M_1+1},\ldots, \tau_\kappa^{M_{\kappa}+1}\right>.
$$
Let $\hat X$ be the analytic space with structure sheaf $\Ok_{\hat X}=\Ok_\U/ \I$.
Consider the tensor product of currents
\begin{equation}\label{hatmu}
\hat\mu=\dbar\frac{d\tau_1}{\tau_1^{M_1+1}}\w\ldots\w\dbar\frac{d\tau_\kappa}{\tau_\kappa^{M_\kappa+1}},
\end{equation}
where $d\tau_j/\tau_j^{M_j+1}$ is the  principal value current.
We recall that if $\varphi = \varphi_0(\zeta,\tau) d\zeta \wedge d\bar{\zeta}$ is a test form, then
\begin{equation}\label{hatmuintegral}
\hat\mu.\varphi=\dbar\frac{d\tau_1}{\tau_1^{M_1+1}}\w\ldots\w\dbar\frac{d\tau_\kappa}{\tau_\kappa^{M_\kappa+1}}.\varphi
=  \frac{(2\pi i)^\kappa}{M!} \int_{\zeta} \frac{\partial \varphi_0}{\partial \tau^M} (\zeta,0) d\zeta \wedge d\bar{\zeta},
\end{equation}
where $M!=M_1!\cdots M_\kappa!$.
It follows, e.g., by \cite[Theorem~4.1]{ACH} that $\hat\mu\w d\zeta$ is a generator for the
$\Ok_\U$-module (and $\Ok_{\hat X}$-module)
$\Homs(\Ok/\I,\CH_\U^Z)$.
For a multiindex $m$, we will use the short-hand notation
\begin{equation} \label{eq:shorthand}
    \dbar\frac{d\tau}{\tau^m} = \dbar\frac{d\tau_1}{\tau_1^{m_1}}\w\ldots\w\dbar\frac{d\tau_\kappa}{\tau_\kappa^{m_\kappa}}.
\end{equation}
Moreover, $m\le M$ means that $m_j\le M_j$ for $j=1,\ldots, \kappa$.
It is readily verified that
\begin{equation} \label{eq:simpleTransformationLaw}
	\tau^\beta \dbar\frac{d\tau}{\tau^\alpha} = \dbar\frac{d\tau}{\tau^{\alpha-\beta}}
\end{equation}
if $\beta\leq \alpha$.
Any $\psi$ in $\Ok_{\hat X}$ has a unique representative in $\U$ of the form
\begin{equation}\label{simple0}
\psi=\sum_{m\le M} \hat\psi_m(\zeta) \tau^m.
\end{equation}
By \cite[Proposition~3.1]{Anorm},
$$
\La_{m,\beta}:=\frac{\partial^{|m|+|\beta|}}{\partial \tau^m \partial \zeta^\beta}\Big|_{\tau=0},
\quad m\le M, \ |\beta|\le |M-m|,
$$
is a generating set for $\N_{\hat X}$.
If $\Psi(\zeta,\tau)$ is any representative in $\U$ for $\psi$, thus, cf.~\eqref{atomat},   
\begin{equation}\label{tomat2}
|\psi|_{\hat X}\sim \sum_{m\le M, \ |\beta|\le |M-m|}
\Big|\frac{\partial\Psi}{\partial \tau^m \partial \zeta^\beta}(\zeta,0)\Big| \sim
\sum_{m\le M, \ |\beta|\le |M-m|} \Big|\frac{\partial\hat{\psi}_m}{\partial \zeta^\beta}\Big|.
\end{equation}
\end{ex}

Let us now return to the setting of a local embedding $i : X \to \U \subset \C^N$ as above.
Notice that if $M$ is large enough in the example, then $\I\subset \J$.
Let $\mu_1,\ldots, \mu_\rho$ be local generators for the coherent $\Ok_\U$-module $\Homs(\Ok_\U/\J,\CH_\U^Z)$.
Then we have a natural mapping $\Ok_\U/\I\to \Ok_\U/\J$, that is, a mapping
$\iota^*\colon \Ok_{\hat X}\to \Ok_X$. It is natural to say that we have an embedding
$$
\iota \colon X\to \hat X.
$$
It is  well-known, see, e.g., \cite[Theorem~1.5]{Aext}
that there are holomorphic functions
$\gamma_1, \ldots, \gamma_\rho$ (possibly after shrinking
$\U$) such that
\begin{equation}\label{tomat3}
\mu_j=\gamma_j \hat\mu, \quad j=1,\ldots,\rho.
\end{equation}
From \cite[Theorem~1.4]{Anorm} we have that
\begin{equation}\label{tomat4}
|\phi|_X\sim \sum_{j=1}^\rho  |\gamma_j \phi|_{\hat X}.
\end{equation}
In this way the norm $|\cdot |_X$ is thus expressed in terms of the simpler norm $|\cdot|_{\hat X}$.

 \subsection{The norm when $X$ is Cohen-Macaulay}\label{putte0}
 So far we have only used the assumption the $Z$ is smooth. Let us now assume in addition that
 $\Ok_X$ is Cohen-Macaulay. Then one can find monomials $1, \tau^{\alpha_1}, \ldots,\tau^{\alpha_{\nu-1}}$
such that each $\phi$ in $\Ok_X$ has a  unique representative
\begin{equation}\label{skata1}
    \hat \phi=\hat{\phi}_0(z)\otimes 1 +\cdots +\hat{\phi}_{\nu-1}(z)\otimes \tau^{\alpha_{\nu-1}},
\end{equation}
where $\hat{\phi}_j$  are in $\Ok_Z$, see, e.g., \cite[Corollary~3.3]{AL}.
In this way $\Ok_X$
becomes  a free $\Ok_Z$-module (in a non-canonical way).
Let $|\cdot |_{X,\pi}$ be the norm obtained from the  subsheaf $\N_{X,\pi}$ of $\N_X$, consisting of
operators $\La$ obtained, cf.~\eqref{tomata},  from the submersion $\pi$ such that  $(\zeta,\tau)\mapsto \zeta$ in $\U$.
It turns out that
\begin{equation}\label{citron1}
|\phi|^2_{X,\pi} \sim |\hat{\phi}_0(z)|^2+\cdots +|\hat{\phi}_{\nu-1}(z)|^2,
\end{equation}
cf.~\cite[Theorem~1.5]{Anorm}.
By \cite[Proposition~3.4]{Anorm}  the whole sheaf $\N_X$ is generated by $\N_{X,\pi}$ for
a finite number of  generic submersions $\pi^\iota$.  It follows that
\begin{equation}\label{citron2}
|\phi|_X\sim \sum_\iota |\phi|_{X,\pi^\iota}.
\end{equation}

\subsection{The sheaf $\E_X^{0,*}$ of smooth forms on $X$} \label{putte1}

Assume that we have a local embedding $i\colon X\to \U\subset\C^N$. If $\Phi$ is in
$\E_\U^{0,*}$ we say that  $i^*\Phi=0$, or equivalently $\Phi$ is in $\Kers i^*$, if  $\Phi$ is in
$\E_\U^{0,*} \J+\E_\Omega^{0,*}\bar \J_Z
+\E_\U^{0,*} d\bar \J_Z$ on $X_{reg}$, where $\J_Z$ is the radical sheaf of $Z$
and we by $X_{reg}$ denote the set of points of $X$ where $Z$ is smooth and $\Ok_X$ is Cohen-Macaulay.

\begin{remark}
If the underlying reduced space $X_{red}$ is not smooth, or $\Ok_X$ is not Cohen-Macaulay, then this definition of
$\Kers i^*$ is not valid. Instead  is used as definition that $\Phi\w \mu=0$ for all
$\mu$ in $\Homs(\Ok_\U/\J,\CH_\U^Z)$.
However, it is true that  $i^*\Phi=0$ if
$i^*\Phi=0$ where $X_{red}$ is smooth and $X$ is Cohen-Macaulay. See
\cite[Lemma~2.2]{AL}.
\end{remark}

We define $\E_X^{0,*}=\E_\U^{0,*}/\Kers i^*$ and have the natural mapping
$i^*\colon \E_\U^{0,*}\to \E_X^{0,*}$.
By standard arguments one can check that the $\Ok_X$-module $\E_X^{0,*}$ so defined
does not depend on the choice of local embedding.

Each $\La\in \N_X$ extends to a mapping $\E_X^{0,*}\to \E_Z^{0,*}$
so that \eqref{tomata} holds, see \cite[Lemma~8.1]{Anorm}.
If we choose a Hermitian metric on the tangent space $TZ$, we get an induced norm on $\E^{0,*}_Z$,
and so we get a pointwise norm by \eqref{atomat} of smooth $(0,*)$-forms on $X$.
In particular, if $\phi$ and $\xi$ are smooth forms on $X$, then
\begin{equation} \label{timesSmooth}
|\xi\w \phi|_X\le C_\xi |\phi|_X,
\end{equation}
cf.~the remark after \cite[(4.23)]{Anorm}.
Choosing an embedding $\iota\colon X\to \hat X$ as above, \eqref{timesSmooth}
follows from \eqref{tomat2} and \eqref{tomat4}.

\smallskip
If $\Ok_X$ is Cohen-Macaulay and $i\colon X \to \U$ is a local embedding with
coordinates $(\zeta,\tau)$ and a monomial basis $\tau^{\alpha_\ell}$, then we have
a unique local representation \eqref{skata1} of each $\phi$ in $\E_X^{0,*}$ with
$\hat\phi_\ell$ in $\E_Z^{0,*}$, and the other statements in Section~\ref{putte1} hold
verbatim, with the same proofs, for smooth $(0,*)$-forms.

\section{Intrinsic currents on $X$}\label{intrinsic}
In the reduced case one can define currents just as dual elements of smooth forms.
In the non-reduced case one has to be cautious because there are two natural kinds of currents;
suitable limits of smooth forms and dual elements of smooth forms. We have to deal with both kinds.
In this paper, the former type appears as $(0,*)$-currents, while the latter appears as $(n,*)$-currents.
In \cite{ALp}, we study the $\dbar$-equation on a non-reduced space for general $(p,q)$-forms,
and then both type of currents appear in arbitrary bidegrees.

\subsection{The sheaf of currents $\Cu_\U^Z$}
Let $\U\subset\C^N$ be an open subset and $Z$ a submanifold as before.
Let $\Cu_\U^Z$ be the $\Ok_\U$-sheaf of all $(N,*)$-currents in $\U$ that are
annihilated by $\bar \J_Z$ and $d \bar \J_Z$. Clearly these currents have support on $Z$.

\begin{lma}\label{allan}
 If $(\zeta,\tau)$ are local coordinates in $\U$ so that $Z=\{\tau=0\}$,
  then each current $\mu$ in $\Cu_\U^Z$
has a unique representation
\begin{equation}\label{palt}
\mu=\sum_{\alpha\ge 0} a_\alpha(\zeta) \wedge \dbar\frac{d\tau}{\tau^{\alpha+\1}}\w d\zeta,
\end{equation}
where $a_\alpha$ are in $\Cu_Z^{0,*}$ and the sum is finite, and we use the short-hand notation \eqref{eq:shorthand}.
\end{lma}

\begin{proof}
Since $\bar \tau_j \mu=0$ for all $j$, $\mu$  must have support on $Z$.
Since it is a current, there is a tuple $M$ of positive integers such that $\tau^\alpha \wedge \mu$
is non-zero only when $\alpha \leq M$.
Let $\pi$ be the
projection $(\zeta,\tau)\mapsto \zeta$. We claim that \eqref{palt} holds with
\begin{equation}\label{eq:aalphaformula}
    a_\alpha(\zeta) \wedge d\zeta = \pm \frac{1}{(2\pi i)^p} \pi_*(\tau^\alpha\mu).
\end{equation}
In fact, given a test form
$\phi$ with Taylor expansion
$$
\phi(\zeta,\tau)=\sum_{\alpha\le M} \frac{1}{\alpha !}\frac{\partial^\alpha \phi}
{\partial \tau^\alpha}(\zeta,0) \tau^\alpha +\Ok(\bar \tau,d\bar \tau) +\Ok(\tau_1^{M_1+1},\dots,\tau_\kappa^{M_\kappa+1}),
$$
and using \eqref{hatmuintegral}, we see that
$$
\mu.\phi=\sum_{\alpha\le M}\mu. \frac{1}{\alpha !}\frac{\partial^\alpha\phi}{\partial \tau^\alpha}(\zeta,0) \tau^\alpha=
\sum_{\alpha\le M} \frac{1}{(2\pi i)^p}\pi_*(\tau^\alpha\mu)\wedge \dbar\frac{d\tau}{\tau^{\alpha+\1}}.\phi.
$$
\end{proof}

It follows from \eqref{eq:aalphaformula} that if $\mu$ has the expansion \eqref{palt}, then
$\dbar \mu$ has an expansion
\begin{equation*}
\dbar\mu=\sum_{\alpha\ge 0} \dbar a_\alpha(\zeta) \wedge \dbar\frac{d\tau}{\tau^{\alpha+\1}}\w d\zeta.
\end{equation*}

In particular,  $\dbar\mu = 0$ if and only if each $a_\alpha(\zeta)$ is $\dbar$-closed.
It is also readily verified that
a sequence $\mu_k$ tends to $0$ if and only if the associated sums
\eqref{palt} have uniformly bounded length and their coefficients $a_{k,\alpha}$ tend to
$0$ for each fixed $\alpha$.

\subsection{The intrinsic sheaf $\Cu_X^{n,*}$}
We define the sheaf $\Cu_X^{n,*}$ of intrinsic $(n,*)$-currents on $X$ as the dual of
$\E_X^{0,n-*}$.
Assume that $i\colon X\to \U$ is a local embedding.  Since
 $\E_X^{0,n-*}=\E_\U^{0,n-*}/\Kers i^*$ and $\Kers i^*$ is closed,
the elements in $\Cu_X^{n,*}$
are represented by the currents in $\U$ that
vanish when acting on test forms with a factor in  $\J, \bar \J_Z, d\bar\J_Z$, which in turn are the currents
in $\U$ that are annihilated by $\J,  \bar \J_Z, d\bar\J_Z$, that is,
$\Homs(\Ok_\U/\J, \Cu_\U^Z)$.
Therefore we have the isomorphism
$$
i_*\colon \Cu_X^{n,*} \stackrel{\simeq}{\to} \Homs(\Ok_\U/\J, \Cu_\U^Z).
$$
Let $\om_X$ be the subspace of $\dbar$-closed elements
in $\Cu_X^{n,0}$. We then obtain the isomorphism
$$
i_*\colon \om_X \stackrel{\simeq}{\to} \Homs(\Ok_\U/\J, \CH_\U^Z).
$$
 In case $X$ is reduced,  $\om_X$ is the
well-known Barlet sheaf of holomorphic $n$-forms on $X$,
cf.~\cite[Section~5]{AL} and \cite{Barlet}.

\subsection{Representations of $\Ok_X$ and $\E_X^{0,*}$}\label{snabel}
Assume that we have a local embedding $i\colon X\to \Omega$.
Notice that we have a well-defined mapping
\begin{equation}\label{pia1}
\Ok_X\to \Homs\big(\Homs(\Ok_\U/\J, \CH_\U^Z),  \Homs(\Ok_\U/\J, \CH_\U^Z)\big),
 \quad \phi\mapsto (\mu \mapsto \phi \mu).
\end{equation}
It follows from \eqref{duality1} that \eqref{pia1} is injective.

\begin{remark}
It is in fact an isomorphism where
$X$ is Cohen-Macaulay (or more generally where $X$ is $S_2$), see \cite[Theorem~7.3]{AL}.
\end{remark}

Let $\mu_1,\ldots, \mu_\rho$  be generators for $\Homs(\Ok_\U/\J, \CH_\U^Z)$,
and consider an element $\Phi$ in
\newline 
$\Homs\big(\Homs(\Ok_\U/\J, \CH_\Omega^Z),  \Homs(\Ok_\U/\J, \Cu_\U^Z)\big)$.
Moreover, let us choose coordinates $(\zeta,\tau)$ in $\U$ as before.
Since each $\Phi(\mu_j)$ is  in $\Homs(\Ok_\U/\J, \Cu_\U^Z)$
it has a unique representation \eqref{palt}, and if we choose $M$ such that
$\tau_1^{M_1+1},\dots,\tau_\kappa^{M_\kappa+1} \in \J$, then the sum only runs over $\alpha \leq M$.
Thus, $\Phi$ is represented by the tuple $\tilde{\Phi} \in (\Cu_Z^{0,*})^r$ consisting
of all the $a_\alpha(\zeta)$ for the currents $\Phi(\mu_j)$, $j=1,\ldots, \rho$.

\smallskip
Now assume that $X$ is Cohen-Macaulay and choose a monomial basis $\tau^{\alpha_\ell}$
as in Section~\ref{putte0}. 
Each $\phi \in \Ok_X$ is then, cf.~\eqref{skata1},  represented by a tuple $\hat{\phi} \in (\Ok_Z)^\nu$.
Thus the mapping \eqref{pia1} defines a holomorphic sheaf morphism
(matrix)  $T\colon (\Ok_Z)^\nu \to (\Ok_Z)^r$.
It is injective by \eqref{duality1}, so $T$ is generically pointwise injective.
In fact, we have, \cite[Lemma~4.11]{AL}:

 \begin{lma} \label{lma:pointwise}
    The morphism $T$ is pointwise injective.
\end{lma}

We can thus (locally) choose a holomorphic matrix $A$ such that
\begin{equation}\label{parra}
0\to \Ok_Z^\nu\stackrel{T}{\to} \Ok_Z^r\stackrel{A}{\to} \Ok_Z^{r'}
\end{equation}
is pointwise exact, and holomorphic matrices $S$ and $B$
such that
\begin{equation} \label{Thomotopy}
    I = TS + BA.
\end{equation}

 \smallskip
In the same way we have a natural mapping
\begin{equation}\label{pia2}
\E_X^{0,*}\to \Homs\big(\Homs(\Ok_\U/\J, \CH_\U^Z),  \Homs(\Ok_\U/\J, \Cu_\U^Z)\big),   \quad \phi\mapsto (\mu \mapsto \phi \wedge \mu).
\end{equation}
If $\phi$ is in $\E^{0,*}_X$, then the coefficients in the expansion \eqref{palt}
of $\phi\w\mu$ are in $\E_Z^{0,*}$ so the image of $\phi$ in \eqref{pia2} is represented by an
element in $(\E_Z^{0,*})^r$.
If $X$ is Cohen-Macaulay we have the unique representation
\eqref{skata1} with $\hat\phi_\ell$ in $\E_Z^{0,*}$  and hence \eqref{pia2} defines an  $\E_Z^{0,*}$-linear morphism
$(\E_Z^{0,*})^\nu\to (\E_Z^{0,*})^r$  that coincides with $T$  for holomorphic $\phi$.
Since \eqref{parra} is pointwise exact, we have the exact complex
\begin{equation}\label{parra2}
0\to (\E_Z^{0,*})^\nu\stackrel{T}{\to} (\E_Z^{0,*})^r\stackrel{A}{\to} (\E_Z^{0,*})^{r'}.
\end{equation}

 \smallskip
We now consider what happens with these representations when we change coordinates.

\begin{lma}\label{picard}
Let $(\zeta,\tau)$ and $(\zeta',\tau')$ be two coordinate systems in $\U$ as before.  There is
a matrix $L$ of holomorphic differential operators such that if $\mu \in \Homs(\Ok_\U/\J,\Cu_\U^Z)$,
and  $(a_\alpha)$ and $(a_\alpha')$ are the coefficients in the associated expansions \eqref{palt}, then
$(a'_\alpha)=L (a_\alpha)$.
\end{lma}

\begin{proof}
Let $\pi$ and $\pi'$ be the projections $(\zeta,\tau)\mapsto \zeta$ and $(\zeta',\tau')\mapsto \zeta'$, respectively.
Fix a multiindex $\alpha$.
Recall that $a_\alpha'\w d\zeta'=\pm (2\pi i)^{-1}\pi'_*((\tau')^\alpha \mu)$, cf.~\eqref{eq:aalphaformula}.
We can write
$$
(\tau')^\alpha=\sum_\rho b_\rho(\zeta)\tau^\rho,
$$
where $b_\rho$ are holomorphic.  After a preliminary change of coordinates in the $\zeta$-variables, which only affects the
coefficients by the factor $d\zeta/d\zeta'$, we may assume that $\zeta'=\zeta$ when $\tau'=0$ so that
$$
\zeta'_j=\zeta_j+\sum_k \tau_k b_{jk}(\zeta,\tau).
$$
If $\varphi=\sum_{|I|=*}' \varphi_I(\zeta) d\bar\zeta_I$ is a test form in $Z$ of bidegree  $(0,*)$, then
\begin{multline*}
(\pi')^* \varphi = \sum_{|I|=*}' \varphi_I(\zeta') d\bar\zeta_I'+\Ok(\bar\tau,d\bar\tau)=\\
\sum_{|I|=*}' \sum_{\gamma,\delta} c_{\gamma,\delta}(\zeta) \frac{\partial^{|\gamma|} \varphi_I}{\partial \zeta^\gamma}(\zeta) d\bar\zeta_I\tau^\delta + \Ok(\bar\tau, d\bar{\tau})=
\sum_{\gamma,\delta} c_{\gamma,\delta}(\zeta) \frac{\partial^{|\gamma|}\pi^* \varphi}{\partial \zeta^\gamma}
\tau^\delta + \Ok(\bar\tau, d\bar{\tau}).
\end{multline*}
Since $\bar\tau\mu=0$ and  $d\bar{\tau} \wedge \mu = 0$, and $b_\rho$ and $c_{\gamma,\delta}$ only depend on $\zeta$,
    \begin{equation*}
        \pi'_*((\tau')^\alpha \mu) . \varphi = 
        \mu .(\tau')^\alpha (\pi')^* \varphi=
\sum_{\rho,\gamma,\delta} \pm \frac{\partial^{|\gamma|}}{\partial\zeta^\gamma} \left( c_{\gamma,\delta} \wedge b_\rho \wedge \pi_*(\tau^{\rho+\delta}\mu)\right) . \varphi
    \end{equation*}
 which means that
\begin{equation}\label{eq:aprimalpha}
 a_\alpha'=\sum_{\rho,\gamma,\delta} \pm \frac{\partial^{|\gamma|}}{\partial\zeta^\gamma} \left( c_{\gamma,\delta} \wedge b_\rho \wedge a_{\rho+\delta} \right).
\end{equation}
Note that the expansion of $(\pi')^*\varphi$ is infinite, but it only runs over $\gamma$ such that $|\gamma|\leq|\delta|$.
Since $\tau^{\delta} \mu = 0$ if $|\delta|$ is large enough, the series \eqref{eq:aprimalpha} defining $a'_\alpha$ is thus in fact a finite sum.
Thus  $a_\alpha'$  is obtained from a matrix of holomorphic differential operators applied to $(a_{\beta})$.
\end{proof}

\begin{cor}\label{ponny}
Assume that $X$ is Cohen-Macaulay.  Let $(\zeta,\tau)$ and $(\zeta',\tau')$ be coordinate systems as above, and let
$\tau^{\alpha_\ell}$ and $(\tau')^{\alpha_\ell'}$ be bases as in Section~\ref{putte0}.
There is a matrix $\mathcal L$ of holomorphic differential operators such that if
$(\hat{\phi}_j)\in (\E^{0,*})^\nu$ and $(\hat{\phi}'_j)\in (\E^{0,*})^\nu$ are the coefficients  with respect to these
two bases of the same element $\phi\in \E^{0,*}_X$,
then
\begin{equation}\label{pommes}
(\hat{\phi}'_j)=\mathcal L (\hat{\phi}_j).
\end{equation}
\end{cor}

\begin{proof}
Consider $\phi\in \E^{0,*}_X$ and let $\Phi$ be its image in
$\Homs\big(\Homs(\Ok_\U/\J, \CH_\U^Z),\Homs(\Ok_\U/\J,\Cu_\U^Z)\big)$.
Given generators $\mu_1,\ldots,\mu_\rho$ for $\Homs(\Ok_\U/\J),\CH_\U^Z)$,
let $T'$ and $S'$ be the mappings $T$ and $S$ but with respect to the new variables
$(\zeta',\tau')$ and basis $(\tau')^{\alpha'_j}$.  Then $T(\hat\phi_j)$ and $T'(\hat\phi'_j)$ are
the coefficients with respect to $(\zeta,\tau)$ and $(\zeta',\tau')$, respectively, of
$\Phi (\mu_j)$, $j=1,\ldots,\rho$.

According to  Lemma~\ref{picard},
then  $T'(\hat\phi'_j) =\tilde L (T(\hat\phi_j))$,  where $\tilde L$ is  a matrix of holomorphic
differential operators. Thus, cf.~\eqref{Thomotopy},
 $(\hat\phi'_j)= S'\circ T'(\hat\phi'_j) =S'\circ \tilde L (T (\hat\phi_j))$, which defines the desired matrix $\La$.
\end{proof}

\subsection{The sheaf $\Cu_X^{0,*}$ of $(0,*)$-currents}
Let us assume now that $X$ is Cohen-Macaulay.  We want
 $\Cu_X^{0,*}$ to be an $\Ok_X$-sheaf extension of $\E_X^{0,*}$
 so that $\E_X^{0,*}$ is dense in a suitable topology.
 The idea is to define a $(0,*)$-current $\phi$ as something that for each choice of
 coordinates $(\zeta,\tau)$ and basis $\tau^{\alpha_\ell}$ as in Section~\ref{putte0} has a representation \eqref{skata1}
where $(\hat\phi_j)$ are in $(\Cu_Z^{0,*})^\nu$,  and transform by \eqref{pommes}.
However, to get a more invariant definition we will represent
$\Cu_X^{0,*}$ as a subsheaf of the $\Ok_X$-sheaf
$$
\F:=\Homs\big(\Homs(\Ok_\U/\J, \CH_\U^Z),\Homs(\Ok_\U/\J, \Cu_\U^Z)\big).
$$
%

Let us fix $(\zeta,\tau,\tau^{\alpha_\ell})$.  Given an expression \eqref{skata1},
where $\hat\phi_0,\ldots\hat\phi_{\nu-1}$ are in $\Cu_Z^{0,*}$, we get a mapping
\begin{equation}\label{sune}
\CH_\U^Z\to \Cu_\U^Z, \quad  \mu\mapsto \hat\phi\w \mu,
\end{equation}
by expressing
$\mu$ as in \eqref{palt} and performing the multiplication
formally term by term.

\begin{lma} \label{lma:Fprim}
The mapping \eqref{sune} defines an element in $\F$ that is zero if and only if
all $\hat\phi_\ell$ vanish.

All such images
    in $\F$ form a coherent subsheaf $\F'$ of $\F$ that is independent of the local
    choice $(\zeta,\tau,\tau^{\alpha_\ell})$.
\end{lma}

\begin{proof}
We first claim that   $\J (\hat\phi \wedge \mu)=0$ if $\J \mu=0$. Let
$(\hat\phi_{\ell,\epsilon})$ be  tuples in $(\E_Z^{0,*})^\nu$ obtained by regularizing each entry $\hat{\phi}_\ell$,
and let $\hat\phi_\epsilon$ denote the corresponding smooth forms in  $\E_X^{0,*}$.
Then  $\hat \phi _\epsilon\w \mu\to \hat \phi\w \mu$ as currents, and since $\J(\hat \phi_\epsilon \wedge \mu)=0$,
the claim follows.  Thus \eqref{sune} defines an element in $\F$.

Let  $\mu_1,\ldots,\mu_\rho$ be generators for $\Homs(\Ok_\U/\J, \CH_\U^Z)$. Then the coefficients of $\hat\phi\w \mu_j$,
$j=1,\ldots,\rho$,  are given by $T(\hat\phi_\ell)\in (\Cu_Z^{0,*})^r$, where
$T$ is the matrix in \eqref{parra2}.  Indeed this holds for the smooth $\hat\phi_\epsilon$, and hence for $\hat\phi$.
Since $T$ is pointwise injective, the induced mapping is injective as well. If the image of \eqref{sune} vanishes
therefore the tuple $\hat\phi_\ell$ vanishes.

For each multiindex $\gamma$,
$$
\tau^\gamma=\sum_\ell c_\ell(\zeta)\tau^{\alpha_\ell}
$$
for unique $c_\ell$ in $\Ok_Z$. For $\xi$ in $\Ok_X$
therefore there is a (unique)
matrix $A_\xi$ of $\Ok_Z$-functions such that
\begin{equation}\label{pannkaka}
(\hat\psi_\ell)=A_\xi(\hat\phi_\ell)
\end{equation}
for any smooth $\hat\phi$ if $\psi=\xi\phi$. Moreover,
\begin{equation}\label{pannkaka2}
\xi(\hat\phi\w\mu)=\hat\psi\w\mu
\end{equation}
since both sides are the equal to the current $\xi\phi\w \mu$.
  If now $(\hat\phi_\ell)$ is in $(\Cu_Z^{0,*})^\nu$ and
$(\hat\psi_\ell)$ is defined by \eqref{pannkaka}, then by a regularization as above
we see that still \eqref{pannkaka2} holds.
Thus the image of $(\Cu_Z^{0,*})^\nu$  is a  locally finitely generated $\Ok_X$-module,
and hence a coherent subsheaf  $\F'$ of $\F$.

It remains to check the independence of the choice of $(\zeta,\tau,\tau^{\alpha_\ell})$.
Thus assume $(\zeta',\tau', (\tau')^{\alpha'_\ell})$ is another choice, cf.~Corollary~\ref{ponny}.
If  $(\hat{\phi}'_\ell)=\mathcal L (\hat{\phi}_\ell)$ and $(\hat\phi'_{j,\epsilon})=\mathcal La (\hat\phi_{\ell,\epsilon})$,
then $\hat{\phi_\epsilon'}\to \hat{\phi}$.  Since $\hat\phi'_\epsilon\w\mu=
\hat{\phi_\epsilon}\w \mu$ we conclude that $\hat\phi'\w\mu=
\hat\phi\w \mu$.
\end{proof}

\begin{df}
   The sheaf of $(0,*)$-currents $\Cu_X^{0,*}$ is defined as the sheaf $\F'$. 
\end{df}

Given $(\zeta,\tau,\tau^{\alpha_\ell})$, thus each element $\phi$ in
$\Cu_X^{0,*}$ has a unique representation \eqref{skata1}.
However,
in view of Lemma~\ref{lma:Fprim}, the current   $\phi\w \mu$ has an invariant meaning.
We have natural mappings $\dbar\colon \Cu_X^{0,q}\to \Cu_X^{0,q+1}$, defined
by $(\hat\phi_\ell)\mapsto (\dbar\hat\phi_\ell)$.  They are well-defined since $\dbar$ commutes with the
transition matrices $\La$ in the preceding proof.  We thus get the complex
\begin{equation}\label{cucomplex}
0\to \Ok_X\to\Cu_X^{0,0}\stackrel{\dbar}{\to}\Cu_X^{0,1}\stackrel{\dbar}{\to}\cdots.
\end{equation}

\begin{prop} \label{lma:dbarClosed}
The sheaf complex   \eqref{cucomplex} is exact.
\end{prop}

\begin{proof}
First assume that $\phi$ is in $\Cu_X^{0,0}$.  Given  $(\zeta,\tau,\tau^{\alpha_\ell})$
we then have a unique representation \eqref{skata1} with $\hat{\phi}_j \in \Cu_Z^{0,*}$.
If $\dbar\phi=0$, then all $\dbar\hat{\phi}_\ell = 0$,
    so $\hat{\phi}_\ell \in \Ok_Z$, and thus $\phi \in \Ok_X$.
If $\phi$ is in $\Cu_X^{0,q+1}$, then $\dbar\phi=0$ means that each $\dbar\hat\phi_\ell=0$, and thus
we have local solutions to $\dbar \hat u_\ell=\hat \phi_\ell$ in $\Cu_Z^{0,q}$. It follows that
$u$ defined by $\hat u_\ell$ is a solution to $\dbar u=\phi$.
\end{proof}

\begin{df}
A sequence $\phi_k$ in $\Cu^{0,*}_X$ converges to $\phi$ if
$\phi_k\w \mu\to \phi\w\mu$ for all $\mu$ in
$\Homs(\Ok_\U/\J, \CH_\U^Z)$.
\end{df}

Notice that $\dbar(\phi\w\mu)=\dbar\phi\w\mu$. Thus
$\phi_k\to \phi$ implies that $\dbar\phi_k\to \dbar\phi$.

\begin{lma}
Let $(\zeta,\tau,\tau^{\alpha_\ell})$ be a local basis  in $\U$ and assume that
\begin{equation}\label{tomte}
\phi_k=\sum \hat\phi_{\ell k} \tau^{\alpha_\ell}, \quad \phi=\sum \hat\phi_{\ell} \tau^{\alpha_\ell}.
\end{equation}
The sequence $\phi_k$ in $\Cu^{0,*}_X$ converges to $\phi$ in $\U$ if and only if
 $\hat\phi_{\ell k}\to \hat\phi_\ell$ for each $\ell$.
\end{lma}

It follows that $\E^{0,*}_X$ is dense in $\Cu^{0,*}_X$, since $\E^{0,*}_Z$ is dense in $\Cu^{0,*}_Z$.

\begin{proof}
If $\hat\phi_{\ell k}\to \hat\phi_\ell$ for each $\ell$, then
$\phi_k\w\mu\to \phi\w\mu$.  For the converse,
let us choose a generating set $\mu_1,\ldots, \mu_\rho$ as above.
If $\phi_k\to \phi$, then in particular $\phi_k\w \mu_j\to \phi\w \mu_j$ for each $j=1,\ldots,\rho$.
This means that $T(\hat\phi_{\ell k})\to T(\hat\phi_\ell)$ for each $\ell$.
Since the matrix $T$ is pointwise injective, therefore
$\hat\phi_{\ell k}\to \hat\phi_\ell$ for each $\ell$.
\end{proof}

\begin{remark}
From the very definition, cf.~Section~\ref{putte1}, a sequence $\phi_k\in\E_X^{0,*}$
tends to $0$ at a given point $x$ if and only if
given a small local embedding $i\colon X\to \U$ at $x$
there are representatives $\Phi_k\in \E_\U^{0,*}$
such that $\Phi_k\to 0$ in $\U$.  If $\tau^{\alpha_\ell}$ is a local basis and
$\hat\phi_{\ell k}\to 0$ for each $\ell$, then
$$
\Phi_k(\zeta,\tau):=\sum_\ell \hat\phi_{\ell k}(\zeta)\tau^{\alpha_\ell} \to 0
$$
in $\U$ and hence $\phi_k\to 0$ in $\E^{0,*}(X\cap \U)$. Also the converse is true.
In fact, if $\Phi_k$ are representatives in $\U$ and $\Phi_k\to 0$ in $\U$, then
each of the coefficients of $\Phi_k\w \mu_j$ in the representation \eqref{palt} tend to
$0$ in $\E^{0,*}(Z\cap \U)$ for each $j$. This precisely means that $T(\hat\phi_{\ell k})$
tend to $0$   in $\E^{0,*}(Z\cap \U)$. Since $T$ is pointwise injective this implies
that $\hat\phi_{\ell k}\to 0$ in $\E^{0,*}(Z\cap \U)$ for each $\ell$.
\end{remark}

\begin{remark}
We only define $\Cu^{0,*}_X$ on the part where $Z$ is smooth, as we there need to embed $L^p_{0,*}(X)$ into a larger space
that allows for more flexibility. We do not know what an appropriate definition of $\Cu^{0,*}_X$ would be over the singular
part of $Z$. In \cite{AL}, we introduce a sheaf $\W^{0,*}_X$ of pseudomeromorphic $(0,*)$-currents on $X$
with the so-called standard extension property, also when $Z$ is singular.
On the part where $Z$ is smooth, $\W^{0,*}_X$ is a subsheaf of $\Cu_X^{0,*}$, and consists of currents which admit a representation \eqref{simple0},
where the $\hat{\psi}_m$ are in $\W^{0,*}_Z \subseteq \Cu^{0,*}_Z$.
\end{remark}

\begin{remark}
We do not know if the embedding  $\Cu_X^{0,*}\to \F$
is an isomorphism, i.e.,
if $\F'=\F$. For any $h$ in $\F$ that can be approximated by smooth forms
$h_\epsilon$ in $\F$, it follows as above that $h$ is in $\F'$, but it is not
clear that this is possible for an arbitrary $h$ in $\F$.
An analogous statement for the subsheaf $\W_X^{0,*}$ is indeed true, see \cite[Lemma~7.5]{AL},
but the proof relies on the fact that elements in $\W_Z^{0,*}$ are in a suitable sense generically smooth, and
does not generalize 
to $\Cu_X^{0,*}$.
\end{remark}

\section{$L^p$-spaces}\label{lpspace}
 Assume that $X$ is Cohen-Macaulay and that the underlying manifold $Z=X_{red}$ is smooth. Recall that
 we have chosen a Hermitian metric on $Z$ and let $dV$ be the associated volume form.
 Assume $1\le p<\infty$.
If $K\subset X$ is a compact subset and $\phi$ is in $\E^{0,*}(X)$ then
$$
\left(\int_{K_{red}}|\phi|^p_X dV\right)^{1/p}
$$
is finite and defines a semi-norm on $\E^{0,*}(X)$.
We define the sheaf $L^p_{loc; 0,*}$ as the completion
of $\E^{0,*}_X$ with respect to these semi-norms.
In particular we get the spaces $L^p_{0,*}(K)$ for any
compact subset $K\subset X$.
For a relatively compact open subset $\V\subset\subset X$
we let $L^p_{0,*}(\V)=L^p_{0,*}(\overline \V)$.
Clearly these spaces are independent of the choice
of $|\cdot|_X$ and Hermitian structure on $Z$.
In the same way we define the sheaf $C^{0,*}_X$
as the completion of $\E^{0,*}_X$ with respect to the semi-norms $\sup_K |\phi|_X$.

\begin{prop}
Let $i\colon X\to \U\subset\C^N$ be a local embedding
and let $\V=X\cap \U$.
If $\phi \in L^p_{0,*}(\V)$, then $\phi\in \Cu^{0,*}(\V)$.
Given coordinates and basis
$(\zeta,\tau,\tau^{\alpha_\ell})$,  
each $\phi \in L^p_{0,*}(\V)$ has a unique
representation \eqref{skata1} where
$\hat{\phi}_\ell \in L^p(\V_{red})$.
If $\phi_k\to \phi$ in $L^p_{0,*}(\V)$, then $\hat\phi_{k\ell}\to \hat\phi_\ell$ in $L^p_{0,*}(\V_{red})$ for each $\ell$.
\end{prop}

If $\phi_k$ are smooth and tend to $\phi$ in $L^p_{0,*}(\V)$ it follows that $\phi_k \to \phi$ in $\Cu^{0,*} (\V)$
and hence
\begin{equation}\label{inka}
L^p_{loc;0,*;X}\subset \Cu^{0,*}_X.
\end{equation}

\begin{proof}
If  $\phi\in L^p_{0,*}(\V)$, then by definition there are smooth $\phi_k$ such that
$\|\phi -\phi_k\|_{L^p(\V)}\to 0$.
Since we have unique representations
$$
\phi_k=\sum_\ell \hat\phi_{k\ell}(\zeta) \tau^{\alpha_\ell}
$$
it follows from \eqref{citron1} and \eqref{citron2} that
 $k\mapsto \hat\phi_{\ell k}$ is a Cauchy sequence in $L^p(\V_{red})$ for each $\ell$
and hence converges to a function $\hat\phi_\ell\in L^p(\V_{red})$. Thus
we have a representation \eqref{skata1} for $\phi$, where $\hat\phi_\ell\in L^p(\V_{red})$.
The last statement now follows from \eqref{citron1} and \eqref{citron2}.
\end{proof}

\begin{ex}\label{simple2}
Let $\hat X$ be the space in Example~\ref{simple1} and let $\hat\V=\U'\cap \hat X$,
where $\U'$ is a relatively compact subset of $\U$.
Let $L^{j,p}(\hat\V_{red})$ be the Sobolev space of all $(0,*)$-currents
whose holomorphic derivates up to order $j$ are in $L^p(Z)$.
It follows from \eqref{tomat} that
$L^p(\hat\V)$ can be realized as all expressions of the form \eqref{simple0}, where
$\psi_m\in L^{|M-m|,p}_{0,*}(\hat\V_{red})$.
\end{ex}

For a general Cohen-Macaulay space $X$ there is no such simple way to describe $L^p(X)$ locally in terms of a single
choice of $(\zeta,\tau,\tau^{\alpha_\ell})$.

\begin{remark}\label{ugn}
Assume that $\phi\in \Cu^{0,*}(\V)$ and that its coefficients $\hat\phi_{\iota, \ell}$
with respect to each of the bases
$(\zeta^\iota,\tau^\iota, (\tau^\iota)^{\alpha^{\iota_\ell}})$, cf.~\eqref{citron2},
are in $L^p(\V_{red})$. We  do not know if this implies that
$\phi$ is in $L^p_{0,*}(\V)$.  Consider the coefficients $\hat\phi_{\iota, \ell}$
with respect to a fixed basis $(\zeta^\iota,\tau^\iota, (\tau^\iota)^{\alpha^{\iota_\ell}})$.
If one approximates these coefficients in $L^p$ by smooth forms, then we get convergence in the
norm $|\cdot|_{X,\pi^\iota}$, cf.~\eqref{citron1}. However, there is no
reason to believe that they converge in the other
norms in the right hand side of \eqref{citron2}.  The problem is that the transition
matrix \eqref{pommes} involves derivatives.
\end{remark}

Notice that if we have an embedding $\iota\colon X\to \hat X$ as in Section~\ref{pnorm},
then, with the notation used there,
$$
\|\phi\|_{L^p(X\cap\U)}\sim \sum_j \| \gamma_j \phi\|_{L^p(\hat X\cap \U)}.
$$

\begin{remark}\label{heuristic}
%
Here is a heuristic proof of Theorem~\ref{main}.
For simplicity, let us assume that only two submersions $\pi^1$ and $\pi^2$ are needed in \eqref{citron2}.
Assume that $\phi$ is in $L^p_{0,1}(\V)$.  Then we can find a solution $u^\iota$ in $\Cu^{0,0}(\V)$ to $\dbar u^\iota=\phi$
so that the coefficients with respect to $(\zeta^\iota,\tau^\iota, (\tau^\iota)^{\alpha^{\iota_\ell}})$ of
$u^\iota$ are  in $L^p(\V_{red})$.  This means that $|u^\iota|_{X,\pi^\iota}$ is in
$L^p(\V_{red})$ for each $\iota$.  Now $h=u^2-u^1$ is $\dbar$-closed, thus holomorphic, and hence bounded.
It follows that also $|u^2|_{X,\pi^1}=|u^1+h|_{X,\pi^1}$ is in $L^p(\V_{red})$.  In view of \eqref{citron2} one
might conclude that $u^2$ actually is in $L^p_{0,0}(\V)$
if we disregard the problem pointed out in Remark~\ref{ugn}.
Clearly, this argument breaks down if $\phi$ has bidegree $(0,q+1)$,  $q\ge 1$.
\end{remark}

It is not clear to us if it is possible to make this outline into
a strict argument. In any case, we will prove Theorem~\ref{main} by means of an integral formula from
\cite{AL}.  Besides being a closed formula for a solution, it also
 makes sense  at non-Cohen-Macaulay points, and offers a possibility
to obtain a~priori estimates, cf.~Section~\ref{nonch}.  Hopefully it could lead to
results for general $(0,q)$-forms.

\section{Koppelman formulas on $X$} \label{koppar}

\subsection{Koppelman formulas in $\C^N$}

Let  $\U\subset\C^N$ be a domain, and
let $\U' \subset\subset\U$. Moreover,  let $\delta_\eta$ be contraction
by the vector field
$$
2\pi i\sum_{j=1}^N (\zeta_j-z_j)\frac{\partial}{\partial \zeta_j}
$$
in $\U_\zeta\times\U'_z$ and let $\nabla_\eta=\delta_\eta-\dbar$.
Assume that  $g=g_{0,0}+\cdots +g_{n,n}$ is a smooth form such that
$g_{k,k}$ has bidegree $(k,k)$ and only contains holomorphic differentials with respect
to $\zeta$. We say that $g$ is a weight in $\U$ with respect to $\U'$ if
$\nabla_\eta g=0$ and $g_{0,0}$ is $1$
when $\zeta=z$.  Notice that if $g$ and $g'$ are weights, then $g'\w g$ is again a weight.
The basic observation is that if $g$ is a weight, then
\begin{equation}\label{basic}
\phi(z)=\int g\phi, \quad z\in \U'
\end{equation}
if $\phi$ is holomorphic in $\U$, see, \cite[Proposition~3.1]{Aint}.

If $\U$ is pseudoconvex, following \cite[Example~1]{Aint2},
we can find  a weight $g$, with respect to $\U'$,  with compact support in $\U$,
such that $g$ depends holomorphically on $z$ and has no anti-holomorphic differentials
with respect to $z$.
For our purpose we can assume that these domains are
balls with center at $0\in\U$. Then we can take
\begin{equation}\label{gval}
g=\chi - \dbar\chi \w \frac{\sigma}{\nabla_\eta\sigma}
=\chi - \dbar \chi \wedge \sum_{\ell=1}^N \frac{1}{(2\pi i)^\ell}\frac{\zeta \cdot d\bar\zeta\w (d\zeta\cdot d\bar \zeta)^{\ell-1}}
{(|\zeta|^2-\bar\zeta\cdot z)^\ell},
\end{equation}
where
$$
\sigma= \frac{1}{2\pi i}\frac{\zeta \cdot d\bar\zeta}{|\zeta|^2-\bar\zeta\cdot z}.
$$
Here $\chi$ is a cutoff function in $\U$ that is $1$ in a \nbh of $\overline{\U'}$.  It is
convenient to choose it of the form $\chi=\tilde \chi(|\zeta|^2)$ where $\tilde \chi(t)$ is identically
$1$ close to $0$ and $0$ when $t$ is large.

Elaborating this construction one can obtain Koppelman formulas for $\dbar$.  Let
$$
b=\frac{1}{2\pi i} \frac{\sum_{j=1}^N (\overline{\zeta_j- z_j}) d\zeta_j}{|\zeta-z|^2}
$$
so that $\delta_\eta b=1$ where $\zeta\neq z$, and
\begin{equation}\label{kanin}
B=\frac{\nabla_\eta b}{\dbar b}=b+ b\wedge\dbar b+\cdots+ b\w(\dbar b)^{N-1}
\end{equation}
is the full Bochner-Martinelli form, cf. \cite[Section~2]{Aint}.
Then $\nabla_\eta B=1-[\Delta]'$, where $[\Delta]'$ is the component with full degree in $d\zeta$
of the current of integration over the diagonal $\Delta\subset \U\times\U'$. Now
\begin{equation}\label{holland}
\K\phi=\int_\zeta g\w B\w \phi
\end{equation}
defines integral operators $\E^{0,*+1}(\U)\to \E^{0,*}(\U')$
such that $\phi=\dbar\K\phi+\K(\dbar\phi)$ in $\U'$.
The integral in \eqref{holland} is, by definition,
the pushforward $\pi_*( g\w B \w\phi)$, where $\pi$ is the natural projection
$\U\times\U'\to \U'$.

\subsection{Hefer morphisms}\label{poker}

Let $(E,f)$ be a locally free resolution as in \eqref{krokus}.
As in \cite{Aint2} and elsewhere, we equip $E:=\oplus E_k$ with a superstructure, by splitting
into the part $\oplus E_{2k}$ of even degree
and the part $\oplus E_{2k+1}$ of odd degree. 
An endomorphism $\alpha \in \End(E)$
is even if it preserves the degree, and odd  if it switches the degree.
The total degree $\deg \alpha$ of a form-valued morphism $\alpha$ is the sum of the endomorphism degree
and the form degree of $\alpha$.  For instance, $f$ is an odd endomorphism.
The contraction by $\delta_\eta$ is a derivation (and has odd degree) that takes the total degree into account,
so if $\alpha$ and $\beta$ are two morphisms, then
$\delta_\eta(\alpha \beta) = \delta_\eta \alpha + (-1)^{\deg \alpha} \alpha \delta_\eta \beta$.

In order to construct
division-interpolation formulas with respect to $(E,f)$, in \cite{Aint2} was introduced the notion
of an associated family $H=(H^\ell_k)$ of Hefer morphisms. Here $H^\ell_k$ are holomorphic
$(k-\ell)$-forms with values in $\Hom(E_{\zeta,k}, E_{z,\ell})$ so they are even.
They are connected in the
following way:  To begin with, $H^\ell_k=0$ if $k-\ell< 0$, and
$H_\ell^\ell$ is equal to $I_{E_\ell}$ when $\zeta=z$. In general,
\begin{equation}\label{heferlikhet}
\delta_\eta H^\ell_{k+1}=H_{k}^{\ell} f_{k+1}(\zeta) - f_{\ell+1}(z)H_{k+1}^{\ell+1}.
\end{equation}
Let $R$ and $U$ be the associated currents, see Section~\ref{prel}. 
The basic observation is that $g'=f_1(z)H^1U+H^0R$ is a kind of non-smooth weight so that
if $\Phi$ is holomorphic, then
\begin{equation}\label{basic2}
\Phi(z)=\int_\zeta g'\w g\Phi = f_1(z)\int_\zeta H^1U\w g\Phi+\int_\zeta H^0R\w g \Phi, \quad z\in \U'.
\end{equation}
When defining these integral operators, we tacitly understand that only components of the integrands that
contribute to the integral should be taken into account.

\subsection{Local Koppelman formulas on $X$}  \label{sec:Kdef}
Now assume that our  non-reduced space $X$ is locally embedded in a pseudoconvex domain $\U$.
Let $\V=X\cap\U$ and $\V'=X\cap \U'\subset\subset \V$.
Let $(E,f)$ be a locally free resolution of $\Ok_X$ as in \eqref{krokus}.
Then $R\Phi=0$ if $\Phi=0$, cf.~\eqref{duality}, and hence
\eqref{basic2} is an intrinsic representation formula
$$
\phi(z)=\int_\zeta p(\zeta,z)\phi(\zeta), \quad z\in  \V',
$$
for $\phi\in\Ok(\V' )$.
Following \cite{AS} and \cite{AL}, one can define operators
\begin{equation}\label{kvadd}
\K\phi(z)=\int_\zeta g\w B\w H^0R\w \phi, \quad z\in \V'
\end{equation}
mapping
$(0,*+1)$-forms in $\V$ to $(0,*)$-forms in $\V'$. However, not even in
'good' cases the formula \eqref{kvadd}, as it stands,
produces a form that is smooth in $\U'$,  cf.~\cite[Remark~10.4]{AL},
so the precise definition of $\K\phi$
is somewhat more involved, cf.~\cite[Section~9]{AL}: If $\mu \in \Homs(\Ok_\U/\J,\CH^Z_\U)$ in $\U'$, then
$\mu(z)\w R(\zeta)$ is a well-defined
\pmm current in $\U\times \U'$. Moreover, $B$ is almost semi-meromorphic in  $\U\times \U'$
and smooth outside the diagonal.
Hence $\mu(z)\w g\w B\w H^0R\w \phi$ is well-defined in $\U\times \U'$,
as the limit of
$\mu(z)\w g\w B^\epsilon\w H^0R\w \phi$, where $B^\epsilon=\chi(|\zeta-z|^2/\epsilon)B$,
cf.~\eqref{eq:PMlimits}.
The equality \eqref{kvadd} is to be interpreted as the fact that there is a unique \pmm current $u=\K\phi$ in $\V'$ such that
$$
\mu\w u=\int_\zeta \mu(z)\w g\w B\w H^0R\w \phi,
$$
for all $\mu \in \Homs(\Ok_\U/\J,\CH^Z)$ in $\U'$.
By \cite[Theorem~9.1]{AL} the operators so defined satisfy the Koppelman formula
\begin{equation}\label{koppel}
\phi=\dbar\K\phi+\K(\dbar\phi)
\end{equation}
in $\V'$.
It turns out, \cite[Theorem~10.1]{AL}, that $\K$ maps $\E^{0,*+1}(\V)\to \E^{0,*}(\V')$ if
$Z$ is smooth and $X$ is Cohen-Macaulay.

\begin{remark}
 In general, $\K\phi$ is not necessarily smooth in $\V'$, so
one has to replace $\E^{0,*}_X$ by the sheaves $\A^{0,*}_X$, cf.~Introduction, \cite{AL} and Section~\ref{nonch}.
\end{remark}

Let us now assume that  $Z=X_{red}$ is smooth.
 By shrinking $\U$ we can assume that we have coordinates $(\zeta,\tau)$ in $\U$
as usual, and we let $(z.w)$ be the  corresponding 'output' coordinates in $\U'$.
If in addition $X$ is Cohen-Macaulay we can choose
$(E,f)$ so that the associated free resolution \eqref{krokus} of $\Ok_\U/\J$
has length $\kappa=N-n$.
Then $R$ has just one component $R_\kappa$.
For a smooth $(0,*+1)$-form $\phi$ in $\V$, then
\begin{equation}\label{neptun}
\K\phi(z,w)= \int_{\zeta,\tau} (g\w B)_n \w H_\kappa^0R_\kappa\w \phi, \quad (z,w)\in \V',
\end{equation}
where $B$ is the Bochner-Martinelli form with respect to $(\zeta,\tau; z,w)$,
and $(\ \ )_n$ denotes the component of bidegree $(n,n-*-1)$ in $(\zeta,\tau)$.

\section{Extension of Koppelman formulas to currents}\label{koppext}

We keep the notation from the preceding section.

\begin{prop} \label{algot}
The operator $\K\colon \E^{0,*+1}(\V)\to \E^{0,*}(\V')$ in \eqref{neptun} extends to
an operator $\Cu^{0,*+1}(\V)\to \Cu^{0,*}(\V')$ and the Koppelman formula
\eqref{koppel} still holds in $\V'$.
\end{prop}

The proposition gives a new proof of the exactness of \eqref{cucomplex}.

\begin{proof}
Let us choose a basis $\tau^{\alpha_\ell}$ for $\Ok_X$ in $\U$, as in Section~\ref{putte0}.
If we represent $\phi \in \Cu^{0,*+1}(\V)$ by $\Phi=\sum \hat{\phi}_\ell (\zeta) \tau^{\alpha_\ell}$,
where $\hat{\phi}_\ell(\zeta) \in \Cu^{0,*+1}(Z \cap \U)$, and  regularize each $\hat{\phi}_\ell$
by $\hat{\phi}_\ell^\epsilon$, we obtain smooth
$\Phi^\epsilon$, representing smooth $\phi^\epsilon$ that tend to $\phi$.
Note that the weight $g$ defining $\K$ has support in the $\zeta$-variable in a fixed
compact set $K \subset \U$, and thus $\K\phi^\epsilon$ is defined when $\epsilon$ is small enough.
We want to show that $\K\phi := \lim_{\epsilon \to 0} \K\phi^\epsilon$ is a
well-defined object in $\Cu^{0,*}(\V')$.

By assumption $B$ is of the form \eqref{kanin},  where
\begin{equation}\label{kanin2}
b=\frac{1}{2\pi i}
\frac{\sum_{j=1}^n (\overline{\zeta_j- z_j}) d\zeta_j+ \sum_{i=1}^\kappa (\overline{\tau_i- w_i}) d\tau_i}
{|\zeta-z|^2+|\tau-w|^2}.
\end{equation}
Take $\mu = \mu(z,w) \in \Homs(\Ok_\U/\J,\CH^Z_\U)$.
 Since $R$ is annihilated by  $\bar{\tau}$ and $d\bar{\tau}$, and  $\mu$ is
 annihilated by $\bar{w}$ and $d\bar{w}$,  see Section~\ref{prel},
 we have that
    \begin{equation}\label{midsommar1}
 \mu(z,w)\w   \K\phi^\epsilon  = \mu(z,w)\w \left(\int_{\zeta,\tau} g(\zeta,z) \wedge B(\zeta,z)\wedge H^0_\kappa R_\kappa \wedge \phi^\epsilon\right),
    \end{equation}
where $B(\zeta,z)$ is the Bochner-Martinelli kernel with respect to the variables $\zeta,z$, and $g(\zeta,z)$ only depends
on $\zeta$ and $z$ (provided that it is chosen as in \eqref{gval}, but for $(\zeta,\tau)$ and $(z,w)$, however,
this special choice of $g$ is not important).
More precisely, in view of the representation \eqref{palt} of $R_\kappa$, its action involves
holomorphic derivatives with respect to $\tau$ followed by evaluation at $\tau=0$, cf.~\eqref{hatmuintegral}.
Therefore all terms involving $\bar\tau$ can be cancelled
without affecting the integral. For the same reason all terms involving $\bar w$ disappear.

Therefore  $H$ is the only factor in the integral that depends on $w$. 
Using the expansions of the form \eqref{simple0} of $\phi$ together with the fact that $R_\kappa$ is annihilated by $\J$, and the expansion
\eqref{palt} of $R_\kappa$, and evaluating the $\tau$-integral in the right hand side of \eqref{midsommar1}
we get
$$
\mu(z,w)\w \left(\int_{\zeta} g(\zeta,z) \wedge B(\zeta,z)\wedge \sum_{\ell'=0}^{\nu-1} h_{\ell'}(\zeta,z, w) \hat\phi^\epsilon_{\ell'} \right),
$$
for appropriate holomorphic functions $h_{\ell'}$.
If we express each occurrence of $w$
in the basis $w^{\alpha_\ell}$ as in \eqref{skata1} modulo $\J$ (with $w$ instead of $\tau$)
and using that $\mu$ is annihilated by $\J$, we get
$$
\mu\w \K\phi^\epsilon=
\mu(z,w)\w \sum_{\ell=0}^{\nu-1} w^{\alpha_\ell} \int_{\zeta} g(\zeta,z) \wedge B(\zeta,z)\wedge \sum_{\ell'=0}^{\nu-1} h_{\ell',\ell}(\zeta,z) \hat\phi^\epsilon_{\ell'},
$$
where $h_{\ell',\ell}$ are polynomials in $\zeta,z$.
Thus
 \begin{equation*}
         \mu(z,w)\w \K\phi^\epsilon  = \mu(z,w)\w \sum_\ell \K_\ell(\hat{\phi}^\epsilon) w^{\alpha_\ell},
    \end{equation*}
    where the $\K_\ell (\hat\phi^\epsilon)$  is the result of multiplying the tuple $(\hat\phi_{\ell'}^\epsilon)$
by  a matrix of smooth forms in $\zeta,z$ followed by convolution by the Bochner-Martinelli form $B(\zeta)$.
    Therefore, each limit $\lim_{\epsilon \to 0}\K_\ell(\hat{\phi}^\epsilon) =: \K_\ell(\phi)$ exists in the sense
    of currents on $Z$ and is  independent of the regularization $\hat{\phi}^\epsilon$, and we see that
 $\K(\phi) = \sum_\ell \K_\ell(\phi) w^{\alpha_\ell}=\lim \K(\phi^\epsilon)$ is well-defined.
Since the Koppelman formula holds for $\phi^\epsilon$, it follows that it also holds for
    $\phi$ by letting $\epsilon \to 0$.
\end{proof}

\section{Comparison of Hefer mappings}\label{comparison}

We will use an instance of the following general result.

\begin{lma}\label{kotor}
    Let $a : (\hat E,\hat{f}) \to (E,f)$ be a morphism of complexes, and let $\hat H$ and $H$ denote
 holomorphic  Hefer mappings associated to $(\hat E,\hat{f})$ and $(E,f)$, respectively.
Then (locally) there exist holomorphic $(k-\ell+1)$-forms $C^\ell_k$ with values in $\Hom(\hat E_{\zeta,k},E_{z,\ell})$
such that
\begin{equation}\label{grus1}
C_k^\ell=0, \quad  k<\ell,
\end{equation}
\begin{equation}\label{grus2}
\delta_\eta C_\ell^\ell=H^\ell_\ell a_\ell(\zeta)-a_\ell(z)\hat H^\ell_\ell,
\end{equation}
and
\begin{equation}\label{grus3}
 \delta_\eta C^\ell_k=
 H^\ell_k a_k(\zeta) - a_\ell(z) {\hat H}^\ell_k  -  C^\ell_{k-1} \hat f_k(\zeta)- f_{\ell+1}(z) C^{\ell+1}_k.
\end{equation}
\end{lma}

Here, just as in \cite{LarComp}, we consider $a$ as a morphism in $\End(\oplus (\hat{E}_k \oplus E_k))$,
and thus $a$ is a morphism of even degree, cf.~Section~\ref{poker}.

\begin{proof} Since $H^\ell_\ell$ and $\hat H^\ell_\ell$
are the identity mappings on $E_{\ell,z}$ and $\hat E_{\ell,z}$, respectively, when $\zeta=z$,
one can solve the equation \eqref{grus2} by \cite[Lemma~5.2]{Aint2}. We now proceed by induction
over $k-\ell$. We know the lemma holds if $k-\ell\le 0$ so let us assume that it is proved
for $k-\ell\le m$ and assume $k-\ell=m+1$.
By \cite[Lemma~5.2]{Aint2}, it is then enough to see that the right hand side of \eqref{grus3}
is $\delta_\eta$-closed. To simplify notation we suppress indices and variables.
By \eqref{heferlikhet}, $\delta H=Hf-fH$ and $\delta \hat H = \hat H \hat f - \hat f \hat H$. In addition, $fa=a\hat f$ and since
$f$ is of odd degree, while $a$ is of even degree, $\delta f=-f\delta$ and $\delta a = a \delta$.
We then have, using that $ff=0$ and $\hat f\hat f=0$,
\begin{multline*}
\delta\big(Ha-a\hat H -(C\hat f-f(z) C)\big)=\\
(Hf-fH)a-a(\hat H\hat f-\hat f\hat H)-(Ha-a\hat H-fC)\hat f +f(Ha-a\hat H-C\hat f)
\end{multline*}
and using the relations above it is readily verified that the right hand side vanishes.
\end{proof}

\section{$L^p$-estimates in special cases}\label{orm}

In this section we consider the space  $\hat X$,  $\Ok_{\hat X}=\Ok_\U/\I$,
in Example~\ref{simple1} where, in a local embedding and suitable coordinates $(\zeta,\tau)$ in $\U$,
$\I=\left<\tau^{M+\1}\right>$.
\smallskip

Since $\I$ is a complete intersection, the Koszul complex
provides a
resolution of $\Ok_\U/\I$. That is, if $e_1,\ldots,e_\kappa$ is a nonsense basis for the trivial
vector bundle $\hat E_1\simeq \C^\kappa\times\U$, then the resolution is generated
by $(\hat E, \hat f)$, where $\hat E_k=\Lambda^k \hat E_1$,
each $\hat f_k$ is contraction by
$$
\tau_1^{M_1+1} e_1^* + \cdots + \tau_\kappa^{M_{\kappa}+1} e_\kappa^*,
$$
and  $e_j^*$ is the dual basis.
The associated residue current is
$$
\hat R_\kappa=\dbar\frac{1}{\tau_1^{M_1+1}}\w\ldots\w\dbar\frac{1}{\tau_\kappa^{M_\kappa+1}} \wedge e_1 \wedge \dots \wedge e_\kappa,
$$
see for example \cite[Corollary~3.5]{ACH}.
In $\U\times\U'$ we use the coordinates $(\zeta,\tau; z,w)$.
If
$$h = \frac{1}{2\pi i}\sum_j \sum_{0 \leq \alpha_j \leq M_j} \tau_j^{\alpha_j} w_j^{M_j-\alpha_j} d\tau_j \wedge e_j^*,$$
then it is readily checked that a choice of Hefer forms $\hat H_k^\ell$ is given by contraction by $\wedge^{k-\ell} h$.
In particular,
$$
\hat H^0_\kappa=\pm \frac{1}{(2\pi i)^\kappa}\sum_{0\le \alpha\le M}w^\alpha \tau^{M-\alpha} d\tau_1\w\ldots d\tau_\kappa \wedge (e_1 \wedge \dots \wedge e_\kappa)^*,
$$
where we use the multiindex notation
$
w^{\alpha}=w_1^{\alpha_1}\cdots w_\kappa^{\alpha_\kappa}. $
In particular, with the notation \eqref{eq:shorthand}, and the formula \eqref{eq:simpleTransformationLaw},
$$
\hat H^0_\kappa \hat R_\kappa =
\frac{1}{(2\pi i)^\kappa}\sum_{\mathbf{0} \leq \alpha \leq M} w^\alpha \dbar\frac{d\tau}{\tau^{\alpha+\1}}.
$$
Using the notation from Section~\ref{sec:Kdef} and Section~\ref{koppext}, we consider the operators
\begin{equation} \label{eq:hatKformula}
\hat\K\psi=\int_{\zeta,\tau} g\w B\w \hat H_\kappa^0 \hat R_\kappa\w \psi
\end{equation}
for $\psi\in\E^{0,*+1}(\U\cap \hat X)$.
As was noted in
Section~\ref{comparison},
only the parts of $B$ and $g$ depending on $z,\zeta$ are relevant.
In view of  \eqref{hatmuintegral}
we therefore get
$$
\hat\K\psi(z,w)=\sum_{\alpha\le M} w^\alpha \int_\zeta g(\zeta,z)\w B(\zeta,z)\w \frac{1}{\alpha !}
\frac{\partial \psi}{\partial\tau^\alpha}(\zeta,0).
$$
Since $B(\zeta,z)$ only depends on $\zeta-z$, by a change of variables, we see that
\begin{equation}\label{plupp}
\frac{\partial(\hat \K\psi)}{\partial w^m \partial z^\gamma} (z,0) =
\sum_{\beta'+\beta''+\delta=\gamma} \int_\zeta B(\zeta,z)\w c_{\beta',\beta''} \frac{\partial g}{\partial z^{\beta'}\partial \zeta^{\beta''}}(\zeta,z)\w
\frac{\partial\psi}{\partial \zeta^{\delta} \partial\tau^m} (\zeta,0)
\end{equation}
for appropriate constants $c_{\beta',\beta''}$.
Since $B(\zeta,z)$ is uniformly integrable in $\zeta$ and $z$, and $g$ is smooth, it follows by, e.g., \cite[Appendix~B]{Ran}, that
\begin{equation}\label{motvilja}
    \left\| \frac{\partial\hat \K\psi}{\partial w^m \partial z^\gamma}(z,0)\right\|_{L^p(Z\cap \U')} \lesssim
    \sum_{\delta \leq \gamma} \left\| \frac{\partial \psi}{\partial \zeta^\delta \partial \tau^m}(\zeta,0)\right\|_{L^p(Z \cap \U)}.
\end{equation}
From \eqref{motvilja} and \eqref{tomat2}  it follows that there is a constant $C_p$ such that
\begin{equation}\label{eq:hatKestimate}
\|\hat\K\psi\|_{L^p(\hat X\cap \U')}\le C_p \|\psi\|_{L^p(\hat X\cap \U)}, \quad 1\le p\le  \infty.
\end{equation}

\begin{ex} \label{ex:product}
Let $X=\C^n\times X_0$ be an analytic space which is the product of $\C^n$ with a space $X_0$ whose
underlying reduced space is a single point $0 \in \C^\kappa$, i.e., such that $\Ok_X = \Ok_{\C^n_\zeta\times\C^\kappa_\tau}/\J$, where $\J = \pi^* \J_0$,
and $\J_0 \subset \Ok_{\C^\kappa_\tau}$ is an ideal such that $Z(\J_0) = 0$ and $\pi$ is the projection $\pi(\zeta,\tau) = \tau$.
Note in particular that this includes the basic examples $\hat{X}$ as in Example~\ref{simple1}.

Since the operator $\hat{\K}$ maps $\tau^\alpha$ to $w^\alpha$, it maps $\J$ to $\J_w$, where $\J_w$ denotes the ideal
$\J$ in the $(z,w)$-coordinates. Furthermore, it maps $\bar{\tau}_k$ and $d\bar{\tau}_k$ to $0$, so it
descends to an operator $\hat{\K} : \E^{0,*+1}(\U\cap X) \to \E^{0,*}(\U'\cap X)$.
Note that one may choose $\gamma_1,\dots,\gamma_\rho$ in \eqref{tomat3} that only depend on $\tau$.
Thus, if $\psi \in \E^{0,*+1}(\U\cap X)$, then
\begin{equation} \label{eq:hatKgamma}
	\hat{\K}(\gamma_k(\tau) \psi) = \gamma_k(w) \hat{\K}\psi.
\end{equation}
By \eqref{tomat4}, \eqref{eq:hatKestimate} and \eqref{eq:hatKgamma}, it follows that
\begin{equation}\label{eq:hatKestimate2}
\|\hat\K\psi\|_{L^p(X\cap \U')}\le C_p \|\psi\|_{L^p(X\cap \U)}, \quad 1\le p\le  \infty.
\end{equation}
\end{ex}

We can now prove

\begin{prop}\label{mainhatx}
Let $X$ be a space of the form $\C^n \times X_0$ as in Example~\ref{ex:product}.
The operators $\K\colon \E^{0,q+1}(X\cap \U) \to \E^{0,q}(X\cap \U')$ extend to bounded operators
$L^p_{0,q+1}(X\cap \U)\to L^p_{0,q}(X\cap \U')$, $q\ge 0$, $1\le p<\infty$,  so that the Koppelman formula
\eqref{koppel} holds.

\smallskip
\noindent
The same statements hold for $C^{0,q}$ instead of $L^p_{0,q}$.
\end{prop}

In particular, if $\psi\in L^p_{0,q+1}(X\cap \U)$ and $\dbar\psi=0$, then
$\dbar \K\psi=\psi$ in $X\cap \U'$ by \eqref{koppel}.
Thus Theorem~\ref{main} holds for all $q$ when $X$ is of the form as in Example~\ref{ex:product}.

\begin{proof}
If $\psi \in L^p_{0,q+1}(X\cap \U)$, then by definition there is a sequence $\psi_k\in
\E^{0,q+1}(X\cap \U)$ such that $\|\psi-\psi_k\|_{L^p(X\cap \U)}\to 0$. It follows from
\eqref{eq:hatKestimate} that $\K\psi_k$ is a Cauchy sequence in $L^p_{0,q}(X\cap \U')$ and hence
has a limit $\K\psi$. Clearly this limit satisfies \eqref{eq:hatKestimate2}.
Moreover, it is in $\Cu^{0,q}(X\cap \U')$. Thus these extended operators satisfy
the Koppelman formula, see Proposition~\ref{algot}.
The statements about $C^{0,q}$ follow in exactly the same way.
\end{proof}

\begin{remark}
We use the intrinsic integral formulas on $\hat X\cap \U$ here for future reference. To obtain the theorem one can just as well
solve the $\dbar$-equation with relevant $L^p$-Sobolev norms in $X\cap \U$ for each coefficient in the expansion \eqref{simple0}.  However, this is naturally done by an integral
formula on $Z\cap \U$, and the required computations are basically the same.
\end{remark}

\medskip

We finish this section with an example showing that the spaces in Example~\ref{ex:product} may not necessarily
be written in the simple form as in Example~\ref{simple1} after a change of coordinates,
even if $\J$ is a complete intersection.

\begin{ex}
Let $\J$ be generated by $(w_1^3,w_1^2+w_2^3)$. Then we claim that one cannot find
local coordinates $\tau_1,\tau_2$ near $0$ such that $\J$ is generated by $(\tau_1^\ell,\tau_2^m)$.
Indeed, since the multiplicity of $\J$ is $9$, $\ell m = 9$. The assumptions imply that
\begin{equation*}
\left[\begin{array}{c} w_1^3 \\ w_1^2+w_2^3 \end{array}\right] =
 \left[ \begin{array}{cc} b_{11} & b_{12} \\ b_{21} & b_{22} \end{array}\right]
 \left[\begin{array}{c} \tau_1^\ell \\ \tau_2^m  \end{array}\right]  \text{ and }
 \tau_j = a_{j1} w_1 + a_{j2} w_2, \text{ for $j=1,2$,}
\end{equation*}
where the $a_{jk}$ and $b_{jk}$ are holomorphic. One may exclude the case $\ell=m=3$ since the above equations would imply that
$w_2^2$ belongs to the ideal generated by $(w_1,w_2)^3$. The case $\ell=1,m=9$ may be excluded as that would
imply that $\tau_1=c_1 w_1^3 + c_2 (w_1^2+w_2^3)$ for some holomorphic functions $c_j$, which would contradict the fact
that $\tau_1$ is part of a coordinate system near $0$.
\end{ex}

\section{$L^p$-estimates at Cohen-Macaulay points} \label{sect:MainProof}

Assume that we have a local embedding $X\to \U$ where $Z\cap \U$ is smooth and $X$
is Cohen-Macaulay. Moreover, assume that we have coordinates $(\zeta,\tau)$ in $\U$ such that
$Z=\{\tau_1=\dots=\tau_\kappa=0\}$, and a basis $\tau^{\alpha_\ell}$ for $\Ok_X$ over $\Ok_Z$.
We may also assume that we have a Hermitian resolution $(E,f)$ of $\Ok_X=\Ok_\U/\J$ of minimal length,
so that its associated residue current is $R=R_\kappa$.

In general, if $X$ is Cohen-Macaulay, and the underlying space $Z$ is smooth, it is not possible to choose coordinates so
that $X$ becomes a product space as in Example~\ref{ex:product}, even if the space is defined by a complete intersection.

\begin{ex}
	Let $\J \subset \Ok_{\C^3_{z,w_1,w_2}}$ be generated by $g=(w_1^2,zw_1+w_2^2)$, and $\Ok_X = \Ok/\J$.
	Then $Z(\J) = \{ w = 0 \}$, so $\J$ is a complete intersection ideal, and $X$ is Cohen-Macaulay.
	We claim that one cannot choose new local coordinates $(\zeta,\tau_1,\tau_2)$ near $0$ such that $\J = \pi^* \J_0$,
	where $\J_0 \subseteq \C^2_\tau$ is an ideal such that $Z(\J_0) = \{ \tau = 0 \}$ and $\pi(\zeta,\tau) = \tau$.

	Indeed, assume that there are such coordinates. First of all, from any set of generators of an ideal, one may
	select among them a minimal subset of generators, and the number is independent of the choice of generators.
	Thus, one may assume that $\J$ is generated by $f_1(\tau),f_2(\tau)$.
	Since $f$ and $g$ generate $\J$, there is an invertible matrix $A$ of holomorphic functions such that $f = Ag$ and $g = A^{-1}f$.
	Note that if $\mathfrak{m}$ is the maximal ideal of functions vanishing at $\{ z=w=0 \}$, then $g$ belongs to
	$\mathfrak{m} \J_Z$. Since $f = A^{-1}g$, the same must hold for $f$.
	Since $\{ \tau = 0 \} = \{ w = 0 \}$, one may write $\tau = B w$
	for some holomorphic matrix $B$. Note also that since $f$ only depends on $\tau$,
	$f=C\tau \mod \J_Z^2$ for some constant matrix $C$. Since $f$ belongs to $\mathfrak{m} \J_Z$,
	we must have that $C = 0$, i.e., $f = 0 \mod \J_Z^2$.
	Thus, also $g = 0 \mod \J_Z^2$, which yields a contradiction.
\end{ex}

Let us assume that we have coordinates $(\zeta,\tau)$ in $\U$
and choose a simple ideal
$\I$ as in Section~\ref{orm}, such that
$\I\subset \J$, and hence, as in Section~\ref{pnorm},  get the embedding
\begin{equation}\label{iota}
\iota\colon X\to \hat X,
\end{equation}
where $\Ok_{\hat X}=\Ok_\U/\I$.
Let $\V=X\cap \U$ and $\V'=X\cap \U'$ as before and let
$\hat \V=\hat X\cap \U$ and $\hat\V'=\hat X\cap \U'$.
Here is our principal result.

\begin{prop}\label{gurka0}
Let $\V$ and $\V'$ be as above and $\K$ as in Section~\ref{sec:Kdef}.

\smallskip
\noindent
(i) There are constants $C_p$, $1 \leq p \le \infty$, such that if $\phi$ is a smooth
$(0,1)$-form and  $\dbar \phi = 0$, then
    \begin{equation}\label{gurka1}
        \|\mathcal{K}\phi \|_{L^p(\V')} \le C_p \|\phi \|_{L^p(\V)}.
    \end{equation}

\smallskip
\noindent
(ii)   If $\phi$ is in $L^p_{0,1}(\V)$, $p<\infty$,
 and $\dbar\phi=0$. Then $\K\phi$ is in $L^p_{0,0}(\V')$,
$\dbar \K\phi=\phi$,  and \eqref{gurka1} holds. If $\phi\in C_{0,1}(\V)$ and
$\dbar\phi=0$, then $\K\in C_{0,0}(\V')$, $\dbar\K\phi=\phi$, and
$$
 \|\mathcal{K}\phi \|_{C(\V')} \le C_\infty \|\phi \|_{C(\V)}.
 $$
 \end{prop}

Clearly Theorem~\ref{main} follows from this proposition.
The rest of this section is devoted to its proof.

\begin{proof}
Choose an embedding \eqref{iota} as above.
Since the proposition is local we can assume that we have a basis $\tau^{\alpha_\ell}$ in $\U$.
Let $\phi$ be a smooth $(0,*)$-form in $\V$.
As in Section~\ref{orm}, let $(\hat E,\hat{f})$ be the Koszul complex of $\I=\langle \tau^{M+\1}\rangle$ in $\U$.
Let us choose a morphism  $a \colon (\hat E, \hat f) \to (E,f)$ of complexes that extends
the natural surjection $\Ok_\U/\I \to \Ok_\U/\J$ and such that
$a_0$ is the identity morphism
$ \hat{E}_0 \simeq E_0$,
see, e.g., \cite[Proposition~3.1]{LarComp}.
By \eqref{tomat4}, we are to estimate the $L^p(\hat \V')$-norm of
$$
\gamma\K\phi=\gamma(z,w)\int_{\zeta,\tau} g\wedge B \wedge H^0_\kappa R_\kappa \wedge \phi,
$$
where $\gamma$ is any of the functions in \eqref{tomat3}.
(By the way, one can choose $\gamma_j$ as the components of $a_\kappa$, cf.~\cite[Example~6.9]{AL}).

Since $\gamma\K\phi$ is to be considered as an element in $\E^{0,*}(\hat\V')$, it is
determined by  $\hat\mu\w \gamma \K\phi$, where
$$
\hat\mu(z,w)=\dbar\frac{dw}{w^{M+\1}}\w dz.
$$
since $\hat\mu$ is a generator for $\Homs(\Ok_\U/\I, \CH_\U^Z)$ in $\U$,
cf.~Section~\ref{sec:Kdef}.

To $\phi$ we associate the representative
$\Phi = \sum \hat{\phi}_\ell(\zeta) \tau^{\alpha_\ell}$ in $\E^{0,*}(\U)$, where $\hat{\phi}_\ell$ are in $\E^{0,*}(Z\cap\U)$, as in \eqref{skata1}.

\begin{lma}\label{palett}
We have that
\begin{equation}\label{eq:palett}
    \hat\mu\w\gamma\K\phi= \hat\mu\w \gamma\int_{\zeta,\tau}  g\wedge B \wedge (\hat H_\kappa^0 + \delta_\eta C_\kappa^0)\hat R_\kappa\wedge
    \Phi.
\end{equation}
\end{lma}

\begin{proof}
Recall from Section~\ref{sec:Kdef} that  $ \hat\mu\w\gamma\K\phi$ is defined as the limit of
    \begin{equation}\label{afrika}
       \hat\mu\w\gamma \int_{\zeta,\tau}  \chi_\epsilon g\wedge B \wedge H^0_\kappa R_\kappa \wedge \Phi,
    \end{equation}
    where $\chi$ is a cut-off function and $\chi_\epsilon = \chi(|(\zeta,\tau)-(z,w)|^2/\epsilon)$.
    By \cite[Theorem~4.1]{LarComp}, $R_\kappa a_0 = a_\kappa \hat{R}_\kappa$.
    Using Lemma~\ref{kotor}, the fact that $a_0$ is the identity, and that $\hat{f}_\kappa \hat{R}_\kappa = 0$ by \eqref{eq:fR}, we get
    \begin{equation} \label{eq:HR}
        H^0_\kappa R_\kappa = H^0_\kappa a_\kappa \hat{R}_\kappa = (\hat{H}^0_\kappa + \delta_\eta C^0_\kappa) \hat{R}_\kappa + f_1(z,w) C^1_\kappa \hat{R}_\kappa.
    \end{equation}
    Since $\gamma \J \subseteq \I$ and $\hat \mu$ is annihilated by $\I$
    we have  that $\gamma(z) f_1(z,w) \hat\mu = 0$ so by \eqref{eq:HR},  \eqref{afrika} is equal to
    \begin{equation*}
\hat\mu\w\gamma\int_{\zeta,\tau}  \chi_\epsilon g\wedge B \wedge (\hat{H}^0_\kappa  + \delta_\eta C^0_\kappa) \hat{R}_\kappa \wedge \Phi.
    \end{equation*}
    Taking the limit as $\epsilon \to 0$, we obtain \eqref{eq:palett}.
\end{proof}

Let us choose  a holomorphic $1$-form $\Gamma$ in $\U$ such that
\begin{equation}\label{gammatrams}
\delta_\eta \Gamma = \gamma(\zeta,\tau)-\gamma(z,w).
\end{equation}
From \eqref{eq:palett} and  \eqref{gammatrams} we get 
\begin{align*}
\hat\mu\w\gamma\K\phi=
\hat\mu\w \int_{\zeta,\tau} g\wedge B \wedge (\hat H_\kappa^0+\delta_\eta C_\kappa^0)\hat R_\kappa \wedge \gamma \phi \\
+\hat\mu\w\int_{\zeta,\tau} g\wedge B \wedge (\hat H_\kappa^0+\delta_\eta C_\kappa^0)\hat R_\kappa\wedge \delta_\eta \Gamma \Phi
=:
\hat\mu \w T_1\phi+ \hat\mu\w T_2\phi.
\end{align*}
Notice that we can write $\phi$ rather than $\Phi$ in $\hat\mu\w T_1\phi$, since
$\hat R_\kappa \gamma$ annihilates $\J$.
Now $T_1\phi=T_{11}\phi+T_{12}\phi$,
where
$$
T_{11}\phi=\int_{\zeta,\tau} g\wedge B \wedge \hat H_\kappa^0\hat R_\kappa \wedge \gamma \phi
$$
and
$$
T_{12}\phi= \int_{\zeta,\tau} g\wedge B \wedge (\delta_\eta C_\kappa^0) \hat R_\kappa )\wedge \gamma\phi.
$$

\begin{lma} \label{lma:intByParts}
    Let $A$ be a holomorphic $(\kappa+1,0)$-form in $d\zeta,d\tau$, $\psi = \psi(\zeta,\tau)$ a smooth  $(0,*)$-form on $\U$.
        Then
    \begin{align*}
        \hat\mu\w \int_{\zeta,\tau} g\wedge B \wedge (\delta_\eta A) \hat{R}_\kappa \wedge \psi
        = \hat\mu\w\int_{\zeta,\tau} g \wedge A \hat{R}_\kappa\wedge \psi  \\
        - \hat\mu\w\int_{\zeta,\tau} g\wedge B \wedge A\hat{R}_\kappa \wedge \dbar \psi
        -\hat\mu\w\dbar_{z,w} \int_{\zeta,\tau} g\wedge B \wedge A \hat{R}_\kappa \wedge \psi .
    \end{align*}
\end{lma}

\begin{proof}
    As in the proof of Lemma~\ref{palett},
    \begin{align*}
       \hat\mu\w  \int_{\zeta,\tau} g\wedge B \wedge (\delta_\eta A) \hat{R}_\kappa \wedge \psi
        = \lim_{\epsilon \to 0} \hat\mu\w \int_{\zeta,\tau} \chi_\epsilon g\wedge B \wedge (\delta_\eta A) \hat{R}_\kappa \wedge \psi.
    \end{align*}
    Let $(\ )_{ k}$ denote the component of degree $k$ in $d\zeta,d\tau$. Then
     \begin{align} \label{eq:nablaStrom}
        \begin{gathered}
		(\nabla_\eta(\chi_\epsilon g\wedge B \wedge A  \hat{R}_\kappa \wedge \psi ))_{N} = \\
	-\dbar_{\zeta,\tau}\chi_\epsilon \wedge (g \wedge B)_{n-1} \wedge A \hat{R}_\kappa \wedge \psi  +
        \chi_\epsilon g_{n-1} \wedge A  \hat{R}_\kappa \wedge \psi  \\ -
        \chi_\epsilon (g\wedge B)_{n} \wedge \delta_\eta A  \hat{R}_\kappa \wedge \psi   -
        \chi_\epsilon (g\wedge B)_{n-1} \wedge A \hat{R}_\kappa \wedge \dbar\psi  \\ -
	\dbar_{z,w}\big( \chi_\epsilon (g\wedge B)_{n-1} A \hat{R}_\kappa \wedge \psi\big),
        \end{gathered}
    \end{align}
    where we have used that $\kappa+n=N$,
 $\nabla_\eta g = 0$ since $g$ is a weight, $\chi_\epsilon \nabla_\eta B = \chi_\epsilon$,
    $\nabla_\eta A = \delta_\eta A$ since $A$ is holomorphic, $\hat R_\kappa$ is $\dbar$-closed $(0,\kappa)$-current
    so that
    $\nabla_\eta \hat R_\kappa = 0$, and finally that $g$, $B$ and $A$ are the only terms containing differentials in $d\zeta,d\tau$, and $A$ has degree     $\kappa+1$ in $d\zeta,d\tau$.

\smallskip
    We  claim that
    \begin{equation} \label{eq:claimTendsToZero}
        \lim_{\epsilon \to 0}
        \hat\mu\w \dbar\chi_\epsilon \wedge (g \wedge B)_{n-1} \wedge  A \hat{R}_\kappa \wedge \psi  = 0.
    \end{equation}
    In fact, let us write $B = \sum B_k$. Since $B$ has only holomorphic differentials in $d\zeta,d\tau$,
 $B_k$ has bidegree $(k,k-1)$, and so
       \begin{equation*}
        (g \wedge B)_{n-1} \wedge  A = \sum_{k \leq n-1} g_{n-k-1} \wedge B_k \wedge  A.
    \end{equation*}
    In particular, it suffices to show that
    \begin{equation*}
        \lim_{\epsilon \to 0} \hat\mu\w \dbar \chi_\epsilon \wedge B_k \wedge \hat{R}_\kappa  = 0
    \end{equation*}
    for $k \leq n-1$.
    The limit of such a term on the left-hand side is a pseudomeromorphic current of bidegree
    $(*,k+2\kappa)$, see the comment after \eqref{eq:PMlimits}. Since the support of $\dbar\chi_\epsilon$ tends to $\Delta$, the limits have support on $\Delta \cap (Z\times Z) \cong Z \cap \{ pt \}$, which has codimension
    $\kappa+(n+\kappa)=n+2\kappa$. By the dimension principle, Proposition~\ref{prop:dim}, therefore
    the limit of each such term is $0$ since $k+2\kappa < n+2\kappa$. Thus the claim holds.

\smallskip
  The lemma follows from the claim by applying $\hat\mu\w\int_{\zeta,\tau}$ to
   \eqref{eq:nablaStrom} and letting $\epsilon \to 0$ since
   $$
    -(\nabla_\eta( \chi_\epsilon g\wedge B \wedge  A \hat{R}_\kappa \wedge \psi \wedge \hat\mu))_{N}=
    \dbar( \chi_\epsilon g\wedge B \wedge  A \hat{R}_\kappa \wedge \psi \wedge \hat\mu)_{N}=
    d( \chi_\epsilon g\wedge B \wedge  A \hat{R}_\kappa \wedge \psi \wedge \hat\mu)_{N}.
    $$
   so that, by Stokes' theorem,
  \begin{equation*}
        \hat\mu\w \int_{\zeta,\tau} (\nabla_\eta( \chi_\epsilon g\wedge B \wedge  A \hat{R}_\kappa \wedge \psi \wedge \hat\mu))_{N} = 0.
  \end{equation*}

  \end{proof}

\bigskip

Using Lemma~\ref{lma:intByParts} with $A=C_\kappa^0$,
we get that $T_{12} \phi= T_{121}\phi+T_{122}\phi+T_{123}\phi$, where
$$
T_{121}\phi= \int_{\zeta,\tau} g_{n-1} \wedge  C_\kappa^0\hat R_\kappa \w  \gamma\phi,
$$
$$
T_{122}\phi=
- \int_{\zeta,\tau} (g\wedge B)_{n-1} \wedge C_\kappa^0\hat R_\kappa \wedge \gamma\dbar \phi
$$
and
$$
T_{123}\phi=\pm \dbar_{z,w}
\int_{\zeta,\tau} (g\wedge B)_{n-1} \wedge C_\kappa^0\hat R_\kappa \wedge \gamma\phi.
$$
Note that since $\hat\mu f_1(z) = 0$ and $\hat{f}_\kappa \hat{R} = 0$, we get that
\begin{align*}
	T_2 \phi = \hat\mu\w\int_{\zeta,\tau} g\wedge B \wedge \delta_\eta ((\hat H_\kappa^0+\delta_\eta C_\kappa^0) \Gamma)\hat R_\kappa\wedge \Phi.
\end{align*}
Thus, by applying Lemma~\ref{lma:intByParts} with $A=(\hat H^0_\kappa+\delta_\eta C^0_\kappa) \wedge \Gamma$, we get that
$T_2\phi = T_{21}\Phi + T_{22}\Phi + T_{23}\Phi$, where
$$
T_{21}\Phi= \int_{\zeta,\tau} g_{n-1,*}\w (\hat H_\kappa^0+\delta_\eta C_\kappa^0)\hat R_\kappa\w \Gamma\w \Phi,
$$
$$
T_{22}\Phi=\int_{\zeta,\tau} g\w B\w (\hat H_\kappa^0+\delta_\eta C_\kappa^0)\hat R_\kappa\w \Gamma\w \dbar\Phi,
$$
and
$$
T_{23}\Phi=\pm\dbar_{z,w}\int_{\zeta,\tau} (g\w B)_{n-1} \w (\hat H_\kappa^0+\delta_\eta C_\kappa^0)\hat R_\kappa\w \Gamma\w \Phi.
$$

\smallskip
We can now prove (i).
If $\dbar\phi=0$, then clearly  $T_{22}\phi$ and $T_{122}\phi$ vanish. If $\phi$ has bidegree $(0,1)$,
then $T_{123}\phi$ and $T_{23}\phi$ vanish for degree reasons since $(g\w B)_{n-1}$ and $\phi$ are the only terms containing $d\bar{\zeta},d\bar{\tau}$.
Therefore,
\begin{equation}\label{part1}
 \gamma \K\phi = T_{11}\phi + T_{121}\phi+T_{21}\Phi.
\end{equation}
The main term $T_{11}\phi$ is precisely $\hat{\K} (\gamma \phi)$,
so from \eqref{eq:hatKestimate} and \eqref{tomat4},
$$
\|T_{11}\phi\|_{L^p(\hat\V')}\le C_p \|\gamma\phi\|_{L^p(\hat \V)}
\le C'_p \|\phi\|_{L^p( \V)}
$$
as desired.
The remaining two terms $T_{121}\phi$ and $T_{21}\Phi$ in \eqref{part1}
are simpler since their integrands do not contain the factor $B$.
We now use that  $\Phi$ has the form \eqref{skata1}
and $\hat R_\kappa$ only depends on $\tau$. Integrating with respect to $\tau$ therefore does not give rise to
any derivates with respect to $\zeta$.
Thus, the $L^p(\V')$-norms of these two terms are bounded by integrals of the form
\begin{equation*}
    \sum_{\ell=0}^{\nu-1}\left(\int_z \left|\int_\zeta |\xi_\ell(\zeta,z) \hat{\phi}_\ell(\zeta)|\right|^p\right)^{1/p},
\end{equation*}
where $\xi_j(\zeta,z)$ are smooth forms with compact support in $Z\cap \U$.  It follows from \eqref{citron1} and \eqref{citron2}
that these terms are $\lesssim \|\phi\|_{L^p(\V)}$.
Thus part (i) is proved.

\smallskip
We now consider part (ii), so assume that $\phi\in L^p_{0,1}(\V)$, $p<\infty$ and  $\dbar\phi=0$.
We cannot deduce (ii) directly from (i).  The problem is that we do not know whether it is possible to
regularize $\phi$ so that the smooth approximands are $\dbar$-closed, cf.~Remarks~\ref{ugn} and \ref{heuristic}.

By Proposition~\ref{algot} we know that
$\dbar\K\phi=\phi$ in the current sense. We must show that actually
$\K\phi$ is in $L^p(\V')$ and that \eqref{gurka1} holds.
Let $\phi_k$ be a sequence of smooth $(0,1)$-forms in $\V$ that converge to $\phi$ in
$L^p(\V)$ and let $\Phi_k$ denote the representatives in $\U$ given by \eqref{skata1}.
Since $T_{123}\phi_k$ and $T_{23}\phi_k$ vanish for degree reasons,
we have
\begin{equation}\label{kulram1}
\gamma\K\phi_k= G\Phi_k+G'(\dbar \Phi_k),
\end{equation}
where
$$
G \Phi_k=  T_{11}\phi_k+ T_{121}\phi_k +T_{21}\Phi_k, \quad
G'_\gamma(\dbar\Phi_k)=T_{122}\Phi_k+T_{22}\Phi_k.
$$
The proof of part (i) gives the a~priori estimate
$$
\|G\tilde \Phi\|_{L^p(\hat\V')}\le C_p \|\tilde \phi\|_{L^p(\V)}
$$
for $\tilde \phi$ in $\E^{0,1}(\V)$.
We conclude that $G\Phi_k$ has a limit $G\Phi$ in $L^p(\hat\V')$
and that
\begin{equation}\label{tennis}
\|G\Phi\|_{L^p(\hat\V')}\le C_p \|\phi\|_{L^p(\V)}
\end{equation}
Next we claim that $\hat\mu\w G'(\dbar\Phi_k)\to 0$.  In fact,
$$
\dbar\Phi_k=\sum_\ell (\dbar\hat\phi_{k,\ell}) \tau^{\alpha_\ell},
$$
so arguing as in the proof of Proposition~\ref{algot} the claim follows, since $\dbar\hat\phi_{k,\ell}\to 0$ for each $\ell$.

Since $\gamma \K\phi_k\to \gamma \K\phi$ in $\Cu^{0,1}(\V')$, it follows from
\eqref{kulram1} that $\gamma \K\phi=G\Phi$.
Thus $\gamma\K\phi$ is indeed in $L^p(\hat\V')$ and, cf.~\eqref{tennis},
$$
\|\gamma\K\phi\|_{L^p(\hat\V')}\le C_p \|\phi\|_{L^p(\V)}.
$$
Since this estimate holds for any $\gamma=\gamma_j$ we get, cf.~\eqref{tomat4},
$$
\|\K\phi\|_{L^p(\V')}\sim \sum_{j=1}^\rho \|\gamma_j\K\phi\|_{L^p(\hat\V')}\le C_p \|\phi\|_{L^p(\V)}.
$$
Thus part (ii) holds for $p<\infty$.
The case $p=\infty$ follows in precisely the same way. Thus the proposition is proved.
\end{proof}

Note that if we drop the assumption that $\phi$ be a $(0,1)$-form, then the terms $T_{123}\phi$ and $T_{23}\phi$ no longer vanish,
and it is not clear to us how to estimate them.
It is also not clear to us whether the estimate \eqref{gurka1} holds if $\phi$ is not $\dbar$-closed.

In the case of product spaces as in Example~\ref{ex:product}, then one may choose $C^0_\kappa$, $\hat H^0_\kappa$ and $\Gamma$ such that they only
contain holomorphic differentials $d\tau$. In that case, all terms but $T_{11}\phi$ vanish for any $(0,q)$-form $\phi$,
since all the other terms involve integrals of forms of degree $\kappa+1$ in $d\tau$, which thus vanish for degree reasons. Thus, one
in fact has that $\gamma \K\phi = T_{11}\phi = \hat{\K}(\gamma \phi)$, cf.\ the proof of Proposition~\ref{mainhatx}.

\section{An example where $X$ is not Cohen-Macaulay}\label{nonch}

In this section we consider an example where $Z=X_{red}$ is smooth but $X$ is not Cohen-Macaulay.
Since $X_{red}$ is smooth, it is still possible to define $L^p_{loc}(X)$ as in Section~\ref{lpspace}.
However, our solutions $\K\phi$ are not smooth at the non-Cohen-Macaulay point.
In view of works on $L^p$-estimates on non-smooth reduced spaces it therefore might be
natural to define $L^p(X)$ as the completion of the space of
smooth forms with support on the Cohen-Macaulay-part of $X$.
In any case we do not pursue this question here, but just discuss an
a~priori estimate of the solutions.

\smallskip

Let $\Omega=\C^4_{z,w}$ and
$\J = \J(w_1^2,w_1 w_2,w_2^2,z_2 w_1 - z_1 w_2)$,
and let $X$ have the structure sheaf $\Ok_\Omega/\J$. Then $Z=\C^2_z$,  and $X$ has the single
non-Cohen-Macaulay point $(0,0)$. Outside that point $X$ is locally of the form discussed in Section~\ref{orm} so
that we have local $L^p$-estimates for $\dbar$ for all $(0,*)$-forms there.
Thus the crucial question is what happens at $(0,0)$.
The structure sheaf $\Ok_X$ has the  free resolution $(E,f)$
    \begin{equation*}
        0 \to \Ok_\Omega \stackrel{f_3}{\longrightarrow} \Ok_\Omega^4 \stackrel{f_2}{\longrightarrow} \Ok_\Omega^4 \stackrel{f_1}{\longrightarrow}
        \Ok_\Omega \to \Ok_\Omega/\J \to 0,
    \end{equation*}
where
    \begin{eqnarray*}
        f_3 = \left[ \begin{array}{c} w_2 \\ -w_1 \\ z_2 \\ -z_1 \end{array} \right] \text{, }
        f_2 = \left[ \begin{array}{cccc} z_2 & 0 & -w_2 & 0 \\ -z_1 & z_2 & w_1 & -w_2 \\ 0 & -z_1 & 0 & w_1 \\ -w_1 & -w_2 & 0 & 0 \end{array} \right] \\
            \text{ and }
            f_1 = \left[ \begin{array}{cccc} w_1^2 & w_1 w_2 & w_2^2 & z_2 w_1 - z_1 w_2 \end{array} \right].
    \end{eqnarray*}
   We equip the vector spaces $E_k$ with the trivial metrics.
Consider also the Koszul complex  $(F,\delta_{\mathbf{w}^2})$ generated by
 $\mathbf{w}^2 := (w_1^2,w_2^2)$, which is a free resolution of $\Ok/ \I$,
 where $\I=\langle w_1^2,w_2^2\rangle$. If $\hat X$ has structure sheaf
 $\Ok_{\hat X}=   \Ok/ \I$ we thus have an embedding
 $\iota\colon X\to \hat X$.

    We  take the morphism of complexes $a : F_\bullet \to E_\bullet$ given by
    \begin{equation*}
        a_2 = \left[ \begin{array}{c} 0 \\ 0 \\ w_2 \\ w_1 \end{array} \right] \text{, }
        a_1 = \left[ \begin{array}{cc} 1 & 0 \\ 0 & 0 \\ 0 & 1 \\ 0 & 0 \end{array} \right]
            \text{ and }
            a_0 = \left[ \begin{array}{c} 1 \end{array} \right].
    \end{equation*}

Let $R$ and $\hat R$ be the residue associated with $(E,f)$ and $(F,\delta_{\mathbf{w}^2)}$, respectively.
It is well-known, see, e.g., \cite{AL}, that $\hat R=\hat R_2$ is equal to the
 Coleff-Herrera product
 $$
 \mu_0 = \dbar(1/w_1^2) \wedge \dbar(1/w_2^2).
$$

\subsection{The current $R$}

In \cite[Example~6.9]{AL} we found that
    \begin{equation*}
        \mu_1 = \dbar \frac{1}{w_1} \wedge \dbar \frac{1}{w_2}
         \text{ and } \mu_2 = (z_1 w_2 + z_2 w_1)\dbar\frac{1}{w_1^2}\wedge \dbar \frac{1}{w_2^2}
    \end{equation*}
 (times $dz\w dw$) generate
$\Homs(\Ok_\Omega/\J,\CH^Z_\Omega)$.  Here we intend to calculate
$R=R_2+R_3$.
Using a comparison with the current $\hat R$
it follows from \cite[Theorem~3.2, Lemma~3.4 and (3.10)]{LarComp} that
\begin{equation} \label{eq:R2exampleFormula}
        R_2 = (I-f_3 \sigma_3) a_2 \mu_0,
    \end{equation}
where
    \begin{equation*}
        \sigma_3 = \frac{1}{|z|^2+|w|^2}\left[ \begin{array}{cccc} \bar{w}_2 & -\bar{w}_1 & \bar{z}_2 & -\bar{z}_1 \end{array} \right]
    \end{equation*}
 is the minimal left-inverse to $f_3$.
 Since $\mu_0$ is pseudomeromorphic with support on $\{ w = 0 \}$,   $\bar{w}_i \mu_0=0$,
 and therefore
    \begin{equation}\label{re2}
        R_2 = \frac{1}{|z|^2}\left[\begin{array}{cccc}
            * & * & -w_2 \bar{z}_2 & w_2 \bar{z}_1 \\
	    * & * & w_1 \bar{z}_2 & -w_1 \bar{z}_1 \\
	    * & * & |z|^2-z_2 \bar{z}_2 & z_2 \bar{z}_1 \\
            * & * & z_1 \bar{z}_2 & |z|^2-z_1 \bar{z}_1
    \end{array}  \right]
    \left[ \begin{array}{c} 0 \\ 0 \\ w_2 \\ w_1 \end{array} \right]
    \mu_0
= \frac{1}{|z|^2} \left[ \begin{array}{c} \bar{z}_1 \mu_1 \\ \bar{z}_2 \mu_1 \\ \bar{z}_1 \mu_2 \\ \bar{z}_2 \mu_2 \end{array}\right].
    \end{equation}
Since $X$ has pure dimension
$R_3 = \dbar \sigma_3 \w R_2$, where the left hand side is the product of the almost semi-meromorphic current
$\dbar\sigma_3$ and the \pmm current $R_2$, cf.~\eqref{eq:PMlimits} and \cite[Section~2]{AL}.
Since $f_3$ is injective, $\sigma_3 = (f_3^* f_3)^{-1} f_3^* = f_3^*/(|z|^2+|w|^2)$.
Thus, $f_3^*(I-f_3 \sigma_3) = 0$, so in view of \eqref{eq:R2exampleFormula},
$R_3 = (|z|^2+|w|^2)^{-1} f_3^* R_2$.
Furthermore, $\bar{w}_j R_2 = d\bar{w}_j \w R_2 = 0$, so we get
\begin{equation}\label{re3}
R_3 = \frac{1}{|z|^2} \left[\begin{array}{cccc} 0 & 0 & d\bar{z}_2 & -d\bar{z}_1 \end{array}\right] R_2
= \frac{\bar{z}_1 d\bar{z}_2 - \bar{z}_2d\bar{z}_1}{|z|^4} \mu_2.
\end{equation}

\subsection{Hefer forms for $(E,f)$}
Recall that a family $H^\ell_k : E_k \to E_\ell$ of
Hefer morphisms are to satisfy, cf.~\eqref{heferlikhet},
$H^\ell_\ell=I_{E_\ell}$ and
\begin{equation}\label{tomat1}
    \delta_{(\zeta,\tau)-(z,w)} H^\ell_k = H^\ell_{k-1} f_k(\zeta,\tau) - f_{\ell+1}(z,w) H^{\ell+1}_k
\end{equation}
for $k > \ell$.
Due to the superstructure, when considering $H$ and $f$ as matrices, \eqref{tomat} means
\begin{equation}\label{tomat}
    \delta_{(\zeta,\tau)-(z,w)} H^\ell_k = H^\ell_{k-1} f_k(\zeta,\tau) - (-1)^{k-\ell-1} f_{\ell+1}(z,w) H^{\ell+1}_k,
\end{equation}
cf.~\cite[(2.12)]{LW2}.
By hands-on calculations, or with the help of Macaulay2, one can check that
\begin{equation*}
    H^0_1=\frac{1}{2\pi i}\left[\begin{array}{c}
       (\tau_1+w_1) d\tau_1+w_1 d\tau_1\\
       \tau_1 d\tau_2+w_2 d\tau_1 \\
       (\tau_2+w_2) d\tau_2 \\
       -\zeta_1 d\tau_2+\zeta_2 d\tau_1+w_1 d\zeta_2-w_2 d\zeta_1
   \end{array}\right]^t,
\end{equation*}
\begin{equation*}
    H^1_2 =
    \frac{1}{2\pi i}\left[\begin{array}{cccc}d\zeta_2&
       0&
       {-d\tau_2}&
       0\\
       {-d\zeta_1}&
       d\zeta_2&
       d\tau_1&
       {-d\tau_2}\\
       0&
       {-d\zeta_1}&
       0&
       d\tau_1\\
       {-d\tau_1}&
       {-d\tau_2}&
       0&
       0\\
       \end{array}\right]
\end{equation*}
\begin{equation*}
    H^2_3=
    \frac{1}{2\pi i}\left[\begin{array}{c}d\tau_2\\
       {-d\tau_1}\\
       d\zeta_2\\
       {-d\zeta_1}\\
   \end{array}\right]
\end{equation*}
\begin{equation*}
    H^0_2=
    \frac{1}{(2\pi i)^2}\left[\begin{array}{cccc}w_1 d\zeta_2 \wedge d\tau_1-w_2 d\zeta_1\wedge d\tau_1 \\
       \zeta_2 d\tau_1\wedge d\tau_2+w_1 d\zeta_2\wedge d\tau_2-w_2 d\zeta_1\wedge d\tau_2 \\
       (\tau_1+w_1) d\tau_1\wedge d\tau_2 \\
       w_2 d\tau_1\wedge d\tau_2
   \end{array}\right]^t
\end{equation*}
\begin{equation*}
    H^1_3=
    \frac{1}{(2\pi i)^2}\left[\begin{array}{c}{-d\zeta_2\wedge d\tau_2}\\
       d\zeta_1\wedge d\tau_2+d\zeta_2\wedge d\tau_1\\
       {-d\zeta_1\wedge d\tau_1}\\
       d\tau_1\wedge d\tau_2\\
       \end{array}\right]
\end{equation*}
\begin{equation*}
    H^0_3=
    \frac{1}{(2\pi i)^3}\left[\begin{array}{c}w_1 d\zeta_2\wedge d\tau_1\wedge d\tau_2-w_2 d\zeta_1\wedge d\tau_1\wedge d\tau_2\\
    \end{array}\right]
\end{equation*}
(where $H^0_1$ and $H^0_2$ are written as transposes of matrices just for space reasons) indeed
satisfy \eqref{tomat} and are thus components of a Hefer morphism.

\subsection{Estimates of integral operators}
Now choose balls $\U'\subset\subset\U\subset\subset \Omega=\C^4_{\zeta,\tau}$ with center at $(0,0)$ and consider
the integral operator
\begin{equation*}
\K\phi=\int_{\zeta,\tau} g \wedge B \w HR\wedge \phi
\end{equation*}
as in Section~\ref{sec:Kdef} for smooth $(0,1)$-forms in $\V'= X\cap \U'$.
We have that $\K\phi=\K_2\phi+\K_3\phi$, where
\begin{equation}\label{k2}
    \mathcal{K}_2 \phi =  \int_{\zeta,\tau} (g_0 B_2+g_1\w B_1) \w H^0_2R_2\wedge \phi
\end{equation}
and
\begin{equation}\label{k3}
    \mathcal{K}_3 \phi(z) = \int_{\zeta,\tau}  \chi    B_1\w H^0_3 R_3\wedge \phi.
\end{equation}
Here $g_0=\chi$ is a cutoff function in
$\U$ with compact support that is equal to $1$ on $\U'$, and $g_{1}$ contains the factor $\dbar\chi$, cf.~\eqref{gval}.
Moreover,
$B_1=b$, $B_2=b\w\dbar b$,
where
$b$ is given by \eqref{kanin2}, cf.~\eqref{kanin}.
Notice however, that since $\bar{\tau} \mu_i = 0$, $d\bar{\tau} \wedge \mu_i = 0$
and that $\bar{w}_i = 0$ considered as a smooth form on $X$, precisely as in Section~\ref{orm},
we can replace $b$ by
$$
\frac{1}{2\pi i} \frac{\sum_{j=1}^2 (\overline{\zeta_j- z_j}) d\zeta_j}{|\zeta-z|^2}
$$
in the formula, and we may assume that $g_0$ and $g_1$ only depend on $\zeta$ and $z$.

For smooth $(0,*)$-forms we have, see \cite[Section~6]{Anorm}, that
\begin{equation}\label{matsnorm}
    |\phi(z,w)|_X \sim |\phi(z,0)|+|z|\left|\frac{\partial}{\partial z}\phi(z,0)\right|+\left| \mathcal L\phi(z,0)\right|,
\end{equation}
where
$$
\mathcal{L} = z_1 \frac{\partial}{\partial w_1}+z_2 \frac{\partial}{\partial w_2}.
$$

Since $B\w g$ has no differentials $d\tau_j$, for degree reasons
we only have to take into account terms of $H$ that contain
the factor $d\tau_1\w d\tau_2$.  By  \eqref{re2}, but with $(\zeta,\tau)$ instead of $(z,w)$,
and the formula above for $H_2^0$ the
relevant part of $(2\pi i)^2 H_2^0 R_2$ therefore is
$$
\frac{1}{|\zeta|^2} \big( |\zeta_2|^2\mu_1+\bar\zeta_1(\tau_1+w_1)\mu_2+\bar\zeta_2 w_2\mu_2\big)=
\mu_1+ \frac{w_1 \overline{\zeta}_1+w_2\overline{\zeta}_2}{|\zeta|^2}\mu_2,
$$
where in the second equality we have used that
$\tau_1 \mu_2 = \zeta_1 \mu_1$.
Thus
 $$
 \K_2\phi= \frac{1}{(2\pi i)^2}  \int_{\zeta,\tau} (g_0 B_2 + g_1\w B_1)\wedge
    \left( \mu_1 + \frac{w_1 \overline{\zeta}_1+w_2\overline{\zeta}_2}{|\zeta|^2}\mu_2  \right) \wedge \phi \wedge d\tau_1 \wedge d\tau_2.
$$
Integrating with respect to $\tau$  and using that
$(2\pi i)^{-2} \mu_2 \wedge \phi \wedge d\tau_1 \wedge d\tau_2 = \mathcal{L} \phi \wedge [\tau = 0]$
we get
\begin{equation} \label{eq:K2}
\K_2\phi =\int_\zeta (g_0 B_2 + g_1\w B_1)\wedge\Big(\phi+ \frac{w_1 \overline{\zeta}_1+w_2\overline{\zeta}_2}{|\zeta|^2}
\w \La \phi\Big).
\end{equation} 

From \eqref{re3} and the formula for $H_3^0$ we get
\begin{align} \label{eq:K3}
    \begin{gathered}
	    \mathcal{K}_3 \phi=        \pm\frac{1}{(2\pi i)^3} \int \chi \frac{w_1\overline{(\zeta_1-z_1)}+w_2\overline{(\zeta_2-z_2)}}{|\zeta-z|^2}\frac{\overline{\zeta_1}d\overline{\zeta}_2-\overline{\zeta}_2 d\overline{\zeta}_1}{|\zeta|^4} \wedge d\zeta_1 \wedge d\zeta_2 \wedge
    \phi \wedge \mu_2 \wedge d\tau_1 \wedge d\tau_2
    = \\ =
    \pm \frac{1}{(2\pi i)^2}\int \chi \frac{w_1\overline{(\zeta_1-z_1)}+w_2\overline{(\zeta_2-z_2)}}{|\zeta-z|^2}\frac{\overline{\zeta_1}d\overline{\zeta}_2-\overline{\zeta}_2 d\overline{\zeta}_1}{|\zeta|^4} \wedge d\zeta_1 \wedge d\zeta_2 \wedge
    (\mathcal{L}\phi)(\zeta,0).
    \end{gathered}.
\end{align}

We now estimate $\K_2\phi$ by considering the various parts of the norm, cf.~\eqref{matsnorm},
letting $K=\supp \chi\cap Z$ and keeping in mind that $z\in X\cap \U'$ so that $|g_1|$ is bounded. To begin with
\begin{equation}\label{k2a}
    |(\mathcal{K}_2 \phi)(z,0)| = \left|\int_\zeta
    (\chi B_2+g_1 B_1) \wedge \phi(\zeta,0) \right| \lesssim
  \int_{\zeta\in K } \frac{1}{|\zeta-z|^3} |\phi(\zeta)|_X.
  \end{equation}
Next we have, cf.~\eqref{plupp},
\begin{equation}\label{k2b}
    |z|\left|\left(\frac{\partial}{\partial z_i} \mathcal{K}_2 \phi\right)(z,0)\right|  = |z|\left| \int B_2 \wedge \frac{\partial}{\partial \zeta_i} \left( \chi \phi(\zeta,0) \right) +\cdots \right| \lesssim
|z| \int_{\zeta\in K} \frac{1}{|\zeta-z|^3}\frac{1}{|\zeta|} |\phi(\zeta)|_X.
\end{equation}
Finally,
\begin{align}\label{k2c}
    |\mathcal{L}\mathcal{K}_2 \phi)(z,0)| =
   \left| \int_\zeta (\chi B_2 +g_1 B_1) \wedge \left( \frac{z_1 \overline{\zeta}_1+z_2\overline{\zeta}_2}{|\zeta|^2} \right) (\mathcal{L} \phi)(\zeta,0) \right|
    \lesssim |z| \int_{\zeta\in K}  \frac{1}{|\zeta-z|^3}\frac{1}{|\zeta|} |\phi(\zeta)|_X.
\end{align}

Since $\K_3\phi$ vanishes when $w=0$, the two first terms in the norm \eqref{matsnorm} vanish, and  thus we get the estimate
\begin{align*}
|\K_3\phi|_X \sim
    |(\mathcal{L} \mathcal{K}_3 \phi)(z,0)| \sim
    \left|\int_{\zeta,z}  \chi \frac{z_1\overline{(\zeta_1-z_1)}+z_2\overline{(\zeta_2-z_2)}}{|\zeta-z|^2}\frac{\overline{\zeta_1}d\overline{\zeta}_2-\overline{\zeta}_2 d\overline{\zeta}_1}{|\zeta|^4} \wedge d\zeta_1 \wedge d\zeta_2 \wedge
    (\mathcal{L}\phi)(\zeta,0)\right| \\
    \lesssim |z|\int_{\zeta\in K} \frac{1}{|\zeta-z|}\frac{1}{|\zeta|^3} |\phi(\zeta)|_X.
       \end{align*}
Thus we have proved
\begin{equation}
|\K_2\phi(z)|_X\le C\int_{\zeta\in K} \Big(1+\frac{|z|}{|\zeta|}\Big)\frac{1}{|\zeta-z|^3},
\quad
|\K_3\phi(z)|_X\le C \int_{\zeta \in K} \frac{1}{|\zeta-z|}\frac{|z|}{|\zeta|^3} |\phi(\zeta)|_X.
\end{equation}

By \cite[Theorem~4.1]{LR2}, $\|\K_2\phi(z)\|_{L^p(\V')} \le C \|\phi\|_{L^p(\V)}$ if $p > 4/3$.
Following the argument of that proof, but where $\|\zeta-z\|^{2n-1}$ is everywhere replaced by $\|\zeta-z\|$,
it follows that $\|\K_3\phi(z)\|_{L^p(\V')} \le C \|\phi\|_{L^p(\V)}$ if $p > 4$.
We thus obtain the following estimate.

\begin{prop}
Let $X$ be the space above and let $\phi$ be a smooth $(0,*)$-form in $\V$.
We have  the a~priori estimate
\begin{equation}\label{tipsy}
\|\K\phi\|_{L^p(\V')}\le C \|\phi\|_{L^p(\V)}
\end{equation}
for $4<p\le \infty$.
\end{prop}

If $\phi$ has bidegree $(0,2)$, then $\K_3\phi$ vanishes for degree reasons,
so then \eqref{tipsy} in fact holds for $p>4/3$.

\end{document}